\documentclass[11pt]{article}
\usepackage[numbers,sort&compress]{natbib}
\usepackage{setspace}
\usepackage{enumerate}
\usepackage{amscd}
\usepackage{amsmath}
\usepackage{latexsym}
\usepackage{amsfonts}
\usepackage{amssymb}
\usepackage{amsthm}
\usepackage{verbatim}
\usepackage{mathrsfs}
\usepackage{enumerate}
\usepackage{hyperref}

\theoremstyle{plain}\newtheorem{definition}{Definition}[section]
\theoremstyle{definition}\newtheorem{theorem}{Theorem}[section]
\theoremstyle{plain}\newtheorem{lemma}[theorem]{Lemma}
\theoremstyle{plain}\newtheorem{coro}[theorem]{Corollary}
\theoremstyle{plain}\newtheorem{prop}[theorem]{Proposition}
\theoremstyle{remark}\newtheorem{remark}{Remark}[section]
\usepackage{xcolor}
\newcommand{\wred}[1]{\textcolor{black}{#1}}

\newcommand{\Div}{\mathrm{div}\,}
\newcommand{\B}{\Big}

\newcommand{\R}{\mathbb{R}}
\newcommand{\be}{\begin{equation}}
\newcommand{\ee}{\end{equation}}
 \newcommand{\ba}{\begin{aligned}}
 \newcommand{\ea}{\end{aligned}}
 
\newcommand{\ffint}{\iint_{Q(r)}\!\!\!\!\!\!\!\!\!\!\!\!\!\!\!\!\!\!\text{---\hspace{-0.03cm}--\,}}
\newcommand{\ffgint}{\iint_{Q(r),\sigma } \!\!\!\!\!\!\!\!\!\!\!\!\!\!\!\!\!\!\!\!\!\text{---\hspace{-0.06cm}--~~}}

\newcommand{\ffgintr}{\iint_{Q(r)} \!\!\! \!\!\!\!\!\!\!\!\!\!\!\!\!\!\text{---\hspace{-0.07cm}--}~}

\newcommand{\ffgintrtietarou}{\iint_{\tilde{Q}(\varrho)} \!\!\! \!\!\!\!\!\!\!\!\!\!\!\!\!\!\text{---\hspace{-0.07cm}--}~}
\newcommand{\ffgintrv}{\iint_{ \vartheta_{1}  r} \!\!\!\!\!\!\!\!\!\!\!\!\!\!\!\text{---\hspace{-0.11cm}--}~}

\newcommand{\ffgintkv}{\iint_{\dot{Q}(\vartheta_{3}) } \!\!\!\!\!\!\!\!\!\!\!\!\!\!\!\!\!\!\!\!\text{-----}~}

\newcommand{\fbxo}{\int_{_{\tilde{B}_{k}} }\!\!\!\!\!\!\!\!\!\!-~\,}

  \newcommand{\f}{\frac}
    
  \newcommand{\ben}{\begin{enumerate}}
   \newcommand{\een}{\end{enumerate}}

\newcommand{\ti}{\nabla}

\newcommand{\Rmnum}[1]{\expandafter\@slowromancap\romannumeral #1@}

\allowdisplaybreaks

\numberwithin{equation}{section}
\begin{document}
%%%%%%%%%%%%%%%%%%%%%%%%%%%%%%%%%%%%%%%%%%%%%%%%%%%%%%%%%%%%%%%%%%%%%%%%%%%%%%%%%%%%%%%%%%%%%%%%%%%%
\title{Partial regularity of suitable weak solutions of the model arising in amorphous molecular beam epitaxy  }
\author{ \;~Yanqing Wang\footnote{ College of Mathematics and Information Science, Zhengzhou University of Light Industry,  Zhengzhou, Henan  450002,  P. R. China Email: wangyanqing20056@gmail.com} ,  ~ Yike Huang\footnote{ College of Mathematics and Information Science, Zhengzhou University of Light Industry, Zhengzhou, Henan  450002,  P. R. China Email: huang\_yike@outlook.com},\;~~Gang Wu\footnote{Corresponding author, School of Mathematical Sciences,  University of Chinese Academy of Sciences, Beijing 100049, P. R. China Email: wugang2011@ucas.ac.cn}~~~and \;~
Daoguo Zhou\footnote{School of Mathematics and Information Sciences, Henan Polytechnic University, Jiaozuo, Henan 454000, P. R. China Email:
zhoudaoguo@gmail.com }
 }
\date{}
\maketitle
\begin{abstract}
In this paper, we are concerned with the precise relationship between the Hausdorff dimension of possible singular point set $\mathcal{S}$ of suitable weak solutions and the parameter $\alpha$ in the nonlinear term in the following parabolic equation
$$h_t+h_{xxxx}+\partial_{xx}|h_x|^\alpha=f.$$
It is shown that  when $5/3\leq\alpha<7/3$, the $\f{3\alpha-5}{\alpha-1}$-dimensional  parabolic Hausdorff measure of $\mathcal{S}$ is zero, which generalizes the recent
corresponding work of Oz\'anski and  Robinson   in
\cite[SIAM J. Math. Anal. 51: 228--255, 2019]{[OR]}  for $\alpha=2$ and $f=0$. The same result  is valid for a       3D  modified Navier-Stokes system.
  \end{abstract}
%%%%%%%%%%
\section{Introduction}
\label{intro}
\setcounter{section}{1}\setcounter{equation}{0}
The fourth-order parabolic equation   describing dynamic crystal growth in materials science is given by
\begin{equation}\label{gsg}
h_t+h_{xxxx}+\partial_{xx}|h_x|^\alpha =f,\quad   \alpha>1.
\end{equation}
Here,   $h$   represent the height of a crystalline layer.
The  surface growth model \eqref{gsg}  plays an important role in molecular-beam-epitaxy (MBE) process and its  physical background  can be found in \cite{[K],[SGG],[BBCGH],[FV],[SW]}.

Considering the
diffusion term $h_{xx}$ due to evaporation-condensation on the left hand side of \eqref{gsg}, Stein and  Winkler \cite{[SW]} constructed the global mild solution with $1<\alpha\leq5/3$ and established the global weak solutions with $5/3<\alpha<10/3$.
 As a special case of equation \eqref{gsg} with $\alpha=2$ and $f=0$,
 \be\label{ckpz}
h_t+h_{xxxx}+\partial_{xx}|h_x|^2=0,
\ee
this equation is known as
the   conserved   Kardar-Parisi-Zhang (Kuramoto-Sivashinsky)  equation.
Recently,
 starting from the work of  Bl\"omer and  Romito   \cite{[BR2009]},
 the mathematical study of equation \eqref{ckpz}
attracts a lot of attention (see, e.g., \cite{[BR2012],[BOS],[O],[OR],[CY]} and  references therein),    since its shares similar features to
the 3D   Navier-Stokes equations.
A celebrated result of the 3D Navier-Stokes system
is that 1-dimensional Hausdorff measure of
   singular set of  its suitable weak solutions    is
   zero. This is so-called Caffarelli-Kohn-Nirenberg theorem \cite{[CKN]}. Notice that  most recent  generalized Caffarelli-Kohn-Nirenberg theorem  \cite{[TY],[RWW],[CDM],[Chen]} mainly   interpret how     the
fractional dissipation
  $(-\Delta)^{\alpha}$ affects
the regularity of suitable
weak solutions  in the 3D Navier-Stokes equations. It seems that  there are few works  involving
how  the   nonlinear terms  affects
the regularity of suitable weak solutions in the Navier-Stokes equations.

We switch  our attention  to the surface growth model.
    In \cite{[OR]}, Oz\'anski and   Robinson first studied the
    partial regularity of suitable weak solution  and successfully
     extended  Caffarelli-Kohn-Nirenberg   theorem to equation  \eqref{ckpz} via the following $\varepsilon$-regularity criterion  involving  dimensionless quantity
\be\label{rc2}
\limsup_{\varrho\rightarrow0}\f{1}{\varrho}\int_{t -\varrho^{4}}^{t +\varrho^{4}} \int^{x +\varrho}_{x -\varrho} | h_{yy}|^{2}dyd\tau \leq\varepsilon,
\ee
 which means that  $h$ is H\"older continuous at point $(x_{0},t_{0})$.
Based on this, it is shown that 1-dimensional Hausdorff measure of  the set of potential singular points
 is zero. Here, a point is said to be a regular point to \eqref{ckpz}  if $h$ is
 H\"older continuous    in some neighborhood
of this point. The rest points will be called singular points and denotes by $\mathcal{S}$.
To this end, they applied the blow-up technology developed by Lin  \cite{[Lin]} and Ladyzenskaja and   Seregin \cite{[LS]}  to the suitable
weak solutions of equation \eqref{ckpz} to establish the $\varepsilon$-regularity criterion   below
\be\label{rc1}
\f{1}{\varrho }\int_{t -\varrho^{4}}^{t +\varrho^{4}} \int^{x +\varrho}_{x -\varrho}| h_y|^{3}dyd\tau \leq\varepsilon.
\ee
Subsequently, the higher regularity of suitable weak solutions under
the $\varepsilon$-regularity  criterion \eqref{rc1} was  obtained by Burczak,  Oza\'nski and  Seregin \cite{[BOS]}. The generalization of
$\varepsilon$-regularity criteria \eqref{rc1} was considered  by Choi and Yang in \cite{[CY]}.

 Inspired by  the  works  \cite{[OR],[SW],[TY],[RWW],[CDM],[Chen]},
we consider the partial regularity in equation \eqref{gsg} to reveal   how  the   nonlinear terms
  affects
the regularity of suitable
weak solutions in equation \eqref{gsg}. We formulate our theorems as follows.%
%Just as $\alpha=2$, equation \eqref{gsg} has the cancellation of the nonlinear term in energy space $L^{2}$, namely,
%$$
%\int\partial_{xx}(|h_{x}|^{\alpha}) hdx=\int  |h_{x}|^{\alpha}h_{xx}  %=
%\f{1}{\alpha+1}\int \partial_{x}(|h_{x}|^{\alpha}h_{x})dx=0
%$$
%$$\ffgintk$$

%\be\label{scaling}
%h_{\lambda}=\lambda^{\f{\alpha-2}{1-\alpha}}h(\lambda^{4}t,\lambda x)
%\ee

\begin{theorem}\label{the1.1}
 Suppose that $h$ is a suitable weak solution to (\ref{gsg}) with $1<\alpha<{7}/{3}$ and force $f$  belonging to  Morrey spaces $\mathcal{M}^{m,\f{\alpha+1}{\alpha}}(Q(1))$, where $m\geq\f{\alpha+1}{\alpha}$. There exist an absolute constant $\varepsilon_{01}$ such that if $h$ satisfy
$$\ba \iint_{Q(1)}|h_y|^{\alpha+1}dyd\tau\leq \varepsilon_{01},  \ea$$
 then  $h$ is H\"older continuous in  $[x-\f{1}{2 },x+\f{1}{2 } ]\times [t-\f{1}{2^{4}},t+\f{1}{2^{4}} ] $ .
\end{theorem}

\begin{theorem}\label{the1.2}
Assume that $h$ is a suitable weak solution  to (\ref{gsg}) with $1<\alpha< {7}/{3}$ and force $f$  is in Morrey spaces  $\mathcal{M}^{m,\f{\alpha+1}{\alpha}}(Q(1))$  with $m\geq\f{\alpha+1}{\alpha}$.  There is a universal constant $\varepsilon_{02}$ such that $h $ is regular  at point $(x,t)$ if
\be\label{con2}
 \limsup_{r\rightarrow
0}  \frac{1}{r^{\f{3\alpha-5}{\alpha-1}}}\iint_{Q( r)}|h_{yy}|^2dyd\tau\leq \varepsilon_{02}.
\ee
\end{theorem}

\begin{remark}
From \eqref{con2}, one immediately yields the regularity of suitable weak solutions of equation \eqref{gsg} with $1<\alpha\leq 5/3$, which is  consistent with  Stein and  Winkler's work \cite{[SW]}.
\end{remark}
The Vitali cover  lemma allows us to estimate the Hausdorff
dimensional of the singular points set of equation \eqref{gsg}.
\begin{coro}\label{coro}
Let $5/3\leq\alpha<7/3$ and $f$ be in Morrey space $\mathcal{M}^{m,\f{\alpha+1}{\alpha}}$  with $m\geq\f{\alpha+1}{\alpha}$.
$\f{3\alpha-5}{\alpha-1}$-dimensional Hausdorff measure of  the set of potential singular points of suitable weak solutions in \eqref{gsg}
 is zero.
\end{coro}
\begin{remark}
Further discussion involving partial regularity of suitable weak solution of equations \eqref{gsg} and \eqref{ckpz}
can be found in Section \ref{section5}.
\end{remark}
Compared with the Oz\'anski and   Robinson's work \cite{[OR]}, the general nonlinear term and no-zero force are considered in above results. The precise relationship between the Hausdorff dimension of possible singular point set $\mathcal{S}$ and the parameter $\alpha$ in the fourth-order  nonlinear term in the  parabolic equation is presented in Corollary \ref{coro}.
It is worth remarking that the  extension of  Poincar\'e inequality of weak  solutions to the  parabolic equation is established. To the knowledge of authors, this type of Poincar\'e inequality of weak  solutions of the  parabolic equation (system) appears frequently  and plays an important role in
partial regularity for  parabolic systems by $\mathcal {A} $-caloric approximation (see \cite{[Scheven],[DMS],[DM],[Struwe],[BF]} and references).  Three kinds of different approaches to prove this type inequality
 are provided. Let $d$ denote the dimension  of time
 direction and  is determined by the parabolic equation (system).
 The first one is partially motivated by the proof of Poincar\'e inequality \cite[Lemma 3.1, p. 458]{[Krylov]} by Krylov
via applying the classical Poincar\'e inequality in time direction. Indeed, following the path of \cite{[Krylov]}, one uses the Poincar\'e inequality  in  time direction to conclude that
$$
\int^{t+r^{d}}_{t-r^{d}}\B|\fbxo u(t)dy -\ffgint u(t)dydt\B|^{p} ds\leq \int^{t+r^{d}}_{t-r^{d}}\B|\fbxo u_{t}(t)dy\B|^{p}d\tau, 1<p<\infty,
 $$
where
$$
\ffgint u(t)dydt= \int^{t+r^{d}}_{t-r^{d}}\fbxo u(t)dyd\tau\B/ \int^{t+r^{d}}_{t-r^{d}} d\tau.
$$
Here, the new ingredient is an  application of the following Poincar\'e-Wirtinger's inequality (see \cite[p. 233]{[Brezis]})
 $$
 \B|\fbxo u(t)dy -\ffgint u(t)dydt\B| \leq C\int^{t+r^{d}}_{t-r^{d}}\B|\fbxo u_{t}(t)dy\B|  d\tau.
 $$
which allows us to slightly improve  the  corresponding Poincar\'e inequality in
\cite{[Krylov]}. It is worth pointing out that an alternative proof of the results in \cite{[Aramaki]} can also be given,  which is of independent interest.
   The second one is modification of that in \cite{[OR]}. The last one originates from   aforementioned works involving  $\mathcal {A} $-caloric approximation \cite{[Scheven],[DMS],[DM],[Struwe],[BF]}.

Inspired by the   equation \eqref{gsg}, we introduce the following modified Navier-Stokes equations
\be\left\{\ba\label{mns}
&{u_{i}}_{t} -\Delta  u_{i}+ u\cdot\ti
u^{\alpha-1}_{i}   +\partial_{i} \Pi= 0, i=1,2,3, \alpha>1,\\
&\Div u =0,\\
&u|_{t=0}= u_0.
\ea\right.\ee
Here,
as the standard Navier-Stokes equations, the system \eqref{mns}  also shares    the cancellation of the nonlinear term in
energy space $L^{2}$, that is,
$$
\int_{\Omega} u_{j}\partial_{j}  u^{\alpha-1}_{i} u_{i}dx=\f{\alpha-1}{\alpha}\int_{\Omega} u_{j}\partial_{j}  u^{\alpha}_{i}  dx=\f{\alpha-1}{\alpha}\int_{\Omega} \partial_{j}(u_{j}  u^{\alpha}_{i})dx =0.
$$
Hence, for the regular solutions of  the   modified Navier-Stokes equations  \eqref{mns}, there holds
the energy equality
  $$
 \|u(T)\|_{L^{2}(\Omega)}^{2}+2 \int_{0}^{T}\|\nabla u\|_{L^{2}(\Omega)}^{2}ds= \|u_0\|_{L^{2}(\Omega)}^{2}.
 $$
The Leary-Hopf type weak solutions of equations \eqref{mns}   can be proved by Galerkin approximation. The next objective is to show how the   nonlinear terms  affects
the regularity of suitable weak solutions in the  modified  Navier-Stokes equations \eqref{mns} .
\begin{theorem}\label{the1.1ns}
Suppose that the pair  $(u, \,\Pi)$ is a suitable weak solution to the modified Navier-Stokes equations   (\ref{mns}) with $1<\alpha< 7/3$.   Then $|u|$  can be bounded by $1$ on $[-\f{1}{8^{2 }},0]\times B(\f{1}{8})$  provided the following condition holds,
\be\label{condns}\ba
\B (\iint_{Q(1)} |u|^{\alpha+1}dydt\B)^{\f{1}{\alpha+1}}
+\B(\iint_{Q (1)}
|\Pi |^{\f{\alpha+1}{\alpha}}dydt \B)^{\f{\alpha}{\alpha+1}}  \leq    \varepsilon_{1}, \ea\ee
for an absolute constant $\varepsilon_{1}>0.$
\end{theorem}
\begin{theorem}\label{the1.2ns}
Suppose that $u$ is a suitable weak solution  to the modified Navier-Stokes equations   \eqref{mns} with $1<\alpha< 7/3$  and for a universal constant {$\varepsilon_{2}>0$},
$$ \limsup_{r\rightarrow
0}\f{1}{r^{\f{5-3\alpha}{1-\alpha}}}\iint_{Q(r)}|\ti u|^{2}dy dt{<\varepsilon_{2}}.
$$
Then  $(0,\,0)$ is a regular point for $u(x,t)$.
\end{theorem}
\begin{coro}\label{corons}
$\f{3\alpha-5}{\alpha-1}$-dimensional Hausdorff measure of  the set of potential singular points of suitable weak solutions in the   modified Navier-Stokes equations  \eqref{mns} with $  5/3\leq\alpha< 7/3$
 is zero.
\end{coro}
 It is coincident that partial regularity results of the   surface growth model \eqref{gsg}  and   the     modified Navier-Stokes equations  \eqref{mns}   are   the same, however,
  the  blow up argument  in \cite{[Lin],[LS],[BR2012],[OR]} can not be directly applied to the modified Navier-Stokes equations due to the different nonlinear term.
   Moreover, it seems that the inductive method developed in \cite{[CKN],[TY],[RWW]} works
 for the modified Navier-Stokes equations  \eqref{mns} with only $\alpha\leq2$ and
 de-Giorgi technique introduced in \cite{[Vasseur]}  breaks down for the general case $\alpha\neq2$.
  We observe that the modified blow up procedure together with the fractional integration theorem \cite{[OLeary],[Kukavica0]} ( Riesz potential estimate \cite{[HLW],[GWZ],[HW],[Wang]} )   involving
parabolic Morrey spaces  allows us to consider the partial regularity of the  modified Navier-Stokes equations  \eqref{mns}.   The proof of Theorem \ref{the1.1ns} is close to
the strategy owing  to  Wang  \cite{[Wang]}  and his co-authors  in \cite{[HLW],[GWZ],[HW],[DHW]}.

It is an interesting question to show the existence of suitable weak solution of the modified Navier-Stokes equations   \eqref{mns}. The  corresponding results at least can be seems as the partial regularity of smooth solution at the first blow-up time. We state the well-posedness of the modified Navier-Stokes equations   \eqref{mns} with  initial data in $H^{1}(\Omega)$.
\begin{theorem}\label{the1.3ns}
Suppose that {$u_{0}\in H^{1}(\Omega)$} and div$u_{0}=0$. There exist a constant $T>0$ and a strong solution of system    \eqref{mns}   with the initial datum $u_{0}$ such that
  $$ L^{\infty}(0,\,T;\,H^{1}(\Omega)) \cap
L^{2}(0,\,T;\,H^{2}(\Omega)).$$
The  strong solution can be extended beyond $t=T$ if
     $$  u  \in L^{p }(0,T; {L ^{q }(\Omega)})  ~\text{with}~  \frac{2 }{p}+\frac{3}{q}=\f{1}{\alpha-1}, ~  q>3\alpha-3.$$
\end{theorem}
 \begin{coro}\label{coro1.3ns}
For every initial {datum} $u_{0}\in H^{1}(\Omega)$ with divergence-free condition, there exists a global regular solution $u$ to modified Navier-Stokes equations \eqref{mns} with  $\alpha\leq5/3$ in
$$
 L^{\infty}(0,\,T;\,H^{1}(\Omega)) \cap
L^{2}(0,\,T;\,H^{2}(\Omega)).
$$
\end{coro}
Finally, we would like to state a result to improve $\varepsilon$-regularity criterion \eqref{rc2} in  conserved Kardar-Parisi-Zhang equation \eqref{ckpz}.
\begin{theorem}\label{the1.3}
There exists an absolute {constant $\varepsilon_{03}$}
with the following property.  If $h$ is a suitable weak solution to (\ref{gsg}) and
\be\label{impro1}
\limsup_{\varrho\rightarrow0}\f{1}{\varrho^{\f{10}{7}}}\int_{t-\varrho^{4}}^{t+\varrho^{4}}\B({\int^{x+\varrho}_{x-\varrho}} |\partial_{yy}h|^{2}dy\B)^{\f67}d\tau\leq\varepsilon_{03},
\ee
or
\be\label{impro2}
\limsup_{\varrho\rightarrow0}\B(\sup_{ -\varrho^{4}\leq \tau-t \leq  \varrho^{4}}\f{1}{\varrho }{\int^{x+\varrho}_{x-\varrho}} |h(y)|dy\B)\leq\varepsilon_{03},
\ee
then $ (x,t)$ is a regular point.
\end{theorem}
\begin{remark}
The $\varepsilon$-regularity criterion \eqref{impro2} means that  the sufficiently small $L^{\infty}_{t,x}$ norm of $h$  yields the regularity of suitable weak solutions of equation \eqref{ckpz}.
This criterion is an improvement of corresponding results in \cite{[OR]}.
\end{remark}

The remaining  paper is structured as follows. In Section 2,
 we present the dimensional analysis,   some dimensionless quantities and auxiliary lemmas  for the surface growth   equation \eqref{gsg}  and the modified Navier-Stokes equations, respectively.
 Then the generalized parabolic Poincar\'e inequality of solutions of equation \eqref{gsg} with three different proof are given. Section 3 is   devoted to the blow-up analysis applying to the equation \eqref{gsg} to show Theorem  \ref{the1.1} and contains the proof of Theorem  \ref{the1.2} and Theorem  \ref{the1.3}.   Section 4, we study the partial regularity of suitable weak solutions of the
 modified Navier-Stokes equations.  In Section 5, we will mention  some problems
  involving partial regularity of suitable weak solution of equations \eqref{gsg} and \eqref{ckpz}.

\section{Notations and  some auxiliary lemmas} \label{section2}

\subsection{Auxiliary results for surface growth model}

It is clear that if    $h (x,t) $  solve  system \eqref{gsg}  with $f(x,t)$, then  $h_{\lambda}   $ is also a solution of \eqref{gsg} with $f_{\lambda}$ for any $\lambda\in  \mathbb{R}^{+} ,$ where
\be\label{eqscaling}
 h_{\lambda}=\lambda^{\f{\alpha-2}{1-\alpha}}h(\lambda x,\lambda^{4}t),f_{\lambda}=\lambda^{\f{2-3\alpha}{1-\alpha}}f(\lambda x,\lambda^{4}t)
\ee
As \cite{[CKN]}, we   can
  assign  a ``dimension" to each quantity as follows
\begin{spacing}{1.35}
$$\begin{tabular}{|c|c|c|c|c|c|c|}\hline
 Quantity&$x $&$t$&$h$&$f$&$\partial_{x}$&$\partial_{t}$ \\ \hline
Dimension &$1$&$4$&$\f{ 2-\alpha}{1-\alpha}$&$\f{ 3\alpha- 2}{1-\alpha}$&$-1$&$-4$\\ \hline
\end{tabular}$$
\end{spacing}
Hence, the quantities in \eqref{rc2}-\eqref{impro2} are all   dimensionless. In addition,
We will use the following quantities:
\begin{align*}
& E_{\ast} (r)= \mathrm{ess\,sup}_{s\in (t-r^4,t+r^4)} \frac{1}{r^{\f{5-3\alpha}{1-\alpha}}}\int_{B(r)} h(s,y)^2\,d y,\\
&  \tilde{E}_{\ast} (r)= \mathrm{ess\,sup}_{s\in (t-r^4,t+r^4)} \frac{1}{r^{\f{5-3\alpha}{1-\alpha}}}\int_{B(r)}\left( h(s)-(h(s))_r \right)^2\,d y,\\
 & E(r)= \frac{1}{r^{\f{3\alpha-5}{\alpha-1}}}\iint_{Q( r)}|h_{yy}|^2dyd\tau,
   \\&D_{p}(r)= \frac{1}{r^{\f{\alpha(p+5)-2p-5 }{\alpha-1}}}\iint_{Q( r)} |h  |^pdyd\tau,\\&\tilde{D}_{p}(r)= \frac{1}{r^{\f{\alpha(p+5)-2p-5 }{\alpha-1}}}\iint_{Q( r)}\B|h-\ffgintr h(t)dydt \B|^pdyd\tau,\\
& E_{\alpha+1}(r)=\frac{1}{r^{\f{2(3-2\alpha)}{1-\alpha}}}
\iint_{Q(r)}|h_y|^{\alpha+1}dyd\tau,\\
&E_{12/7,2}(\varrho)=\f{1}{\varrho^{\f{10}{7}}}\int_{t-\varrho^{4}}^{t+\varrho^{4}}
\B({\int^{x+\varrho}_{x-\varrho}} |\partial_{yy}h|^{2}dy\B)^{\f67}d\tau\\
&D_{\infty,1}(\varrho)=\sup_{ -\varrho^{4}\leq \tau-t \leq  \varrho^{4}}\f{1}{\varrho }{\int^{x+\varrho}_{x-\varrho}} |h(y)|dy,
\end{align*}
where  we used the following notation
$$\ba
&B(x;\mu)=\{y\in \mathbb{R} ||x-y|\leq \mu\},~~ &B&(\mu):= B(x;\mu),
\\
&Q(x,t;\mu)=B(x,\,\mu)\times(t-\mu^{4}, t+\mu^{4}),~~ &Q&(\mu):= Q(x,t;\mu)\\
& \tilde{B}(\mu)=B(x_{0};\,\mu),~~ &\tilde{Q}&(\mu):= Q(x_{0},t_{0};\mu)
,\\
&  \dot{B} (\mu)=B(0;\,\mu),~~ & \dot{Q} &(\mu):= Q(0,0;\mu),
\ea$$
and $$\ba
&h_{_{B(r),\sigma}}(t)={\int_{_{B(r),\sigma} }\!\!\!\!\!\!\!\!\!\!\!\!\!\!\!\!-~\,} ~~~h(t,y)dy =\f{\int_{B(r)} h(y,t)\sigma dy}{\int_{B(r)}\sigma dy},\\
&h_{_{Q(r),\sigma}}=\ffgint~ h(t,y)dy =\f{\iint_{ Q(r)}h(y,t)\sigma dydt}{\iint_{Q(r)}\sigma dydt},\\
&h_{_{B(r) }}(t)=h_{_{B(r),1}}(t),\\
& h_{_{Q(r) }}=h_{_{Q(r),1}},
\ea$$
where  $\sigma(x)$ is a smooth cut off function such that $\sigma(x)=1$ in $B(\vartheta_{1} r)$
and $\sigma(x)=0$ in $B^{c}(r)$, where $0<\vartheta_{1}<1$.

For $p\in [1,\,\infty]$, the notation $L^{p}(0,\,T;X)$ stands for the set of measurable functions $f(x,t)$ on the interval $(0,\,T)$ with values in $X$ and $\|f(\cdot,t)\|_{X}$ belongs to $L^{p}(0,\,T)$. For $f\in L^{1}(\mathbb{T})$, we denote $\hat{f}(k)=\f{1}{2\pi}\int_{\mathbb{T}}e^{-ikx}f(x)dx$ for the
$k$th Fourier coefficients of $f$. The space $H^{s}$
  is equipped with the norm $\|f\|_{H^{s}(\mathbb{T})}=(\sum_{k\in \mathbb{Z}}  ( 1 + |k|^{2s}  ) |\hat{f}(k) |^2 )^{1/2}$. The homogeneous  space $\dot{H}^{s}$  is given by
  $$
  \dot{H}^{s}=\B\{f\B|\|f\|_{\dot{H}^{s}}=(\sum_{k\in \mathbb{ Z}}  |k|^{2s}  |\hat{f}(k) |^2\ )^{1/2}\ \text{and} \int_{\mathbb{T}} f =0 \B\}.
  $$ The Morrey space {$\mathcal{{M}}^{p,l}(\Omega)$}, with $1\leq l<\infty$, $1\leq  p\leq\infty$  and a domain {$\Omega\subset\mathbb{R}^{d}$}, is defined as the space of all measurable functions $f$ on $\Omega$ for which the norm
$$
\|f\|_{{\mathcal{{M}}^{p,l}(\Omega)}}=\sup_{R>0}
\sup_{x\in \Omega}{R^{d(\f1p-\f1l)}}\B(\int_{B_{x}(R)\cap\Omega}|f(y)|^{l}dy\B)^{\f1l}<\infty.
$$
Similarly, one  can define  the
the   parabolic Morrey space  $\mathcal{{M}}^{p,l}(Q(r))$.
We will use $C$ to denote an absolute
   constant which may be different from line to line unless otherwise stated.
   For simplicity,   we write $$\|f\| _{L^{p }L^{q}(Q(r))}:=\|f\| _{L^{p}(t-r^{4},t+r^{4};L^{q}(B(r)))}~~   \text{and}~~
  \|f\| _{L^{p}(Q(r))}:=\|f\| _{L^{p,p}(Q(r))}. $$

\begin{definition}\label{defi}
A  function   $h$  is called a suitable weak solution to the   equation  \eqref{gsg}  provided the following conditions are satisfied,
\begin{enumerate}[(1)]
\item $h \in L^\infty ((0,T);L^2(\mathbb{T})) \cap L^2 ((0,T);\dot{H}^2(\mathbb{T}))$\label{SWS1};
 \item$h$~solves (\ref{gsg})  in the sense of distributions, for $\varphi(x,t)\in C_0^\infty (\mathbb{T} \times (0, T))$,
\be \label{solvesdistri}
\int_{0}^{T}\int_{\mathbb{T}}( h(t)\varphi_{t}  - h_{xx} \varphi_{xx} - |h_x|^\alpha \varphi_{xx} )dxd t =-\int_{0}^{T}\int_{\mathbb{T}}   f  \varphi(t) dxd\tau;
     \ee
\item $h$ satisfies the following inequality, for a.e. $t\in[0,T]$,
\be\label{loc}\ba
&\frac{1}{2} \int_{\mathbb{T}} |h(t)|^2 \phi(t)dx + \int_0^t \int_{{\mathbb{T}}}| h_{xx}|^2\phi dxd\tau\\
\leq& \int_0^t \int_{\mathbb{T}} \left(\frac{1}{2}(\phi_t-\phi_{xxxx})|h|^2 \right.
\left.+2h_x^2\phi_{xx}-
\frac{2\alpha+1}{\alpha+1}|h_x|^{\alpha}h_x\phi_x
-|h_x|^{\alpha}h\phi_{xx} +fh\phi\right)dxd\tau
\ea\ee
holds for all $\phi \in C_0^\infty (\mathbb{T} \times (0,\infty ))$. \label{SWS3}
\end{enumerate}
\end{definition}
\begin{remark}
 Just as $\alpha=2$, equation \eqref{gsg} has the cancellation of the nonlinear term in energy space $L^{2}$, namely,
 $$
 \int_{\mathbb{T}}\partial_{xx}(|h_{x}|^{\alpha}) hdx=\int_{\mathbb{T}}  |h_{x}|^{\alpha}h_{xx} dx =
 \f{1}{\alpha+1}\int_{\mathbb{T}} \partial_{x}(|h_{x}|^{\alpha}h_{x})dx=0.
 $$
 Therefore, the existence of suitable weak solution of equation \eqref{gsg} with $\alpha<7/3$  can be
showed as the argument  in \cite{[OR]}. It seems that the critical case  $\alpha=7/3$  corresponds to the existence of suitable weak solution of  the 4D Navier-Stokes equations.  Maybe a parabolic concentration-compactness method recently  developed by Wu \cite{[Wu]} for the 4D Navier-Stokes equations can deal with this case.
\end{remark}

Next, we are concerned with the Poincar\'e inequality of  weak solutions of the parabolic equation. To the knowledge of authors,
  Poincar\'e type inequality for solutions of parabolic system plays an important role in
  partial regularity for
parabolic systems, especially,  by  $\mathcal {A} $-caloric approximation (see \cite{[Scheven],[DMS],[DM],[Struwe],[BF]} and references).

Let  $\sigma(x)$ be a smooth cut off function such that $\sigma(x)=1$ in $B(\vartheta_{1} r)$
and $\sigma(x)=0$ in  $B^{c}(r)$ , where $0<\vartheta_{1}<1$.
\begin{lemma}\label{parapoincare}
Suppose that $h$  is a weak solution of equation \eqref{gsg} satisfying \eqref{solvesdistri}.
 Then, there hold
 \be\ba\label{poin3}
&\B\|h-h_{_{Q(\vartheta_{1}r) }}\B\|_{L^{p}(Q(\vartheta_{1}r) )}^{p}
\leq C_{1}\B\{ r^{p}\|h_y\|_{L^{p}(Q(r) )}^{p}+r^{5+2p-5\alpha}\|h_{y}\|^{p\alpha}_{L^{p}(Q(r))}+r^{5-p}\|f\|^{p}_{L^{1}(Q(r))}\B\}.
\ea\ee

\end{lemma}
\begin{remark}
The key point for the parabolic Poincar\'e inequality \eqref{poin3} is the following estimate, for $\tau\in (t-r^{4},t+r^{4})$,
\be\ba\label{keypoin}
\B|h_{_{B(r),\sigma}}(\tau) -h_{_{Q(r),\sigma}} \B|
\leq& C r^{-4} \|h_{y}\|_{L^{1}(Q(r))}+Cr^{-3} \|h_{y}\|^{\alpha}_{L^{\alpha}(Q(r))}+Cr^{-1}\|f\|_{L^{1}(Q(r))}.
\ea\ee
Here, we will provide three different methods to show it.
\end{remark}

\begin{proof}   Assume for a while we have proved that \eqref{keypoin}.
By the triangle inequality, the wighted Poincar\'e inequality and \eqref{keypoin}, we know that
$$\ba
&\iint_{Q(r)}\B|h-\ffgint h(t)dydt\B|^{p}\sigma dydt\\
\leq&\iint_{Q(r)} \B|h-\fbxo h(t)dy \B|^{p}\sigma dxdt+\iint_{Q(r)} \B| \fbxo h(t)dy -\ffgint h(t)dydt\B|^{p}\sigma dyd\tau\\
\leq& Cr^{p}\iint_{Q(r)} |h_y|^{p} dydt+ Cr^{5}\B[r^{-4} \|h_{y}\|_{L^{1}(Q(r))}+r^{-3} \|h_{y}\|^{\alpha}_{L^{\alpha}(Q(r))}+r^{-1}\|f\|_{L^{1}(Q(r))}\B]^{p}.
\ea$$
We further deduce  from the H\"older inequality that
$$\ba
&\iint_{Q(r)}\B|h-\ffgint h(t)dxdt\B|^{p}\sigma dxdt\\
\leq& Cr^{p}\iint_{Q(r)} |  h_{y}|^{p} dxdt+Cr^{5+2p-5\alpha}\|h_{y}\|^{p\alpha}_{L^{p}(Q(r))}
 +Cr^{5-p}\|f\|^{p}_{L^{1}(Q(r))}\\
 \leq& Cr^{p}\iint_{Q(r)} |  h_{y}|^{p} dxdt+Cr^{5+2p-5\alpha}\|h_{y}\|^{p\alpha}_{L^{p}(Q(r))}
 +Cr^{5+4p-\frac{5}{m}p}\|f\|^{p}_{\mathcal{M}^{m,1}(Q(1))}.
\ea$$
Direct calculation means that
$$\iint_{Q(\vartheta_{1}r)}\B|h-\ffgintrv h(t)dydt\B|^{p}dxdt\leq C\iint_{Q(\vartheta_{1}r)}\B|h- \ffgint h(t)dydt\B|^{p}\sigma dydt,$$
which yields that desired estimate.

It
suffices to  show \eqref{keypoin}.   The argument can be made rigorous by standard approximation techniques or
by the use of Steklov averages (see \cite{[Scheven],[DMS],[DM],[Struwe]} and references ).

Method (1):
 Note that
$$h_{_{Q(r),\sigma}}=\f{\int_{t-r^{4}}^{t+r^{4}} h_{_{B(r),\sigma}}d\tau}{\int_{t-r^{4}}^{t+r^{4}} d\tau}. $$
Hence, we conclude by the Poincar\'e-Wirtinger's inequality that, for $\tau\in (t-r^{4},t+r^{4})$,
\be\label{2.8}
\B|h_{_{B(r),\sigma}} -h_{_{Q(r),\sigma}} \B|\leq \int_{t-r^{4}}^{t+r^{4}}|h_{_{B(r),\sigma}} ' (\tau) | d\tau.
\ee
By virtue of integration by parts, we arrive at
$$\ba
 h_{_{B(r),\sigma}} ' (\tau) =&\f{\int_{B(r)} h_{\tau}(y,\tau)\sigma dy}{\int_{B(r)}\sigma dx}\\=&\f{\int_{B(r)} (-h_{yyyy}-\partial_{yy}|h_y|^\alpha+f)\sigma dy}{\int_{B(r)}\sigma dx}\\=&\f{\int_{B(r)} ( h_{x }\sigma_{yyy}- |h_y|^\alpha\sigma_{yy }+f\sigma) dy}{\int_{B(r)}\sigma dx}.
\ea$$
The above formal computations can be made rigorous by Steklov averages (see \cite[P.216-218, Proof of Lemma 5.1]{[BF]}).

Plugging this into \eqref{2.8},
we further get
$$
\B|h_{_{B(r),\sigma}} -h_{_{Q(r),\sigma}} \B|
\leq    C r^{-4} \|h_{y}\|_{L^{1}(Q(r))}+Cr^{-3} \|h_{y}\|^{\alpha}_{L^{\alpha}(Q(r))}+Cr^{-1}\|f\|_{L^{1}(Q(r))}.
$$

Method (2): Assume that $\tau_{1},\tau_{2}\in(0,T)$. Without loss of generality, we suppose that $\tau_{1}<\tau_{2}$.
Consider the Lipschitz continuous function which is defined  by
$$
\xi_{\epsilon}(t)=\left\{\ba
& \f{t-\tau_{1}}{\epsilon},~~~\tau_{1} <t\leq\tau_{1}+\epsilon,\\
& 1, ~~~~~~~~~~\tau_{1}+\epsilon<t\leq \tau_{2}-\epsilon,\\
&\f{\tau_{2}-t}{\epsilon}, ~~~~\tau_{2}-\epsilon<t<\tau_{2}.\ea\right.
$$
Let $\varphi(t)=\xi_{\epsilon}(t)(h_{_{B(r),\sigma}}(\tau_{1})-h_{_{B(r),\sigma}}(\tau_{2}) )\sigma(y)$ as the text function in \eqref{solvesdistri}.

According to  Lebesgue's differentiation theorem and  the definition of  $h_{_{B(r),\sigma}}$, we see that
$$\ba
&\lim_{\epsilon\rightarrow0}\int_{0}^{T}\int_{\mathbb{T}}  h(t)\varphi_{t}dydt\\
=&\lim_{\epsilon\rightarrow0}\B(\f{1}{\epsilon}
\int_{\tau_{1}}^{\tau_{1}+\epsilon}\int_{\mathbb{T}}h(t)\sigma dydt-\f{1}{\epsilon}
\int_{\tau_{2}-\epsilon}^{\tau_{2}}\int_{\mathbb{T}}h(t)\sigma dydt
\B)(h_{_{B(r),\sigma}}(\tau_{1})-h_{_{B(r),\sigma}}(\tau_{2}) ) \\
=&\B|\fbxo h(\tau_{1})dy-\fbxo h(\tau_{2})dy\B|^{2}\int_{B(r)}\sigma dy\\
\geq&C r\B|\fbxo h(\tau_{1})dx-\fbxo h(\tau_{2} )dy\B|^{2} .
\ea$$

The H\"older inequality  guarantees that
\be\ba\label{poin5}
 \B|\lim_{\epsilon\rightarrow0}\int_{0}^{T}\int_{\mathbb{T}}  h_{yy} \varphi_{yy}dydt\B|&=\B|(h_{_{B(r),\sigma}}(\tau_{1})-h_{_{B(r),\sigma}}(\tau_{2}) )\int^{\tau_{2}}_{\tau_{1}}\int_{B(r)} h_{y}\sigma_{yyy}dyd\tau\B|\\
&\leq C r^{-3}\B| \fbxo h(\tau_{1})dx-\fbxo h(\tau_{2} )dy\B|\|h_{y}\|_{L^{1}(Q(r))}.\ea\ee
Likewise,
\be\ba
&\B|\lim_{\epsilon\rightarrow0}\int_{0}^{T}\int_{\mathbb{T}}  |h_{y}|^{\alpha}\varphi_{yy}dydt\B|\leq C r^{-2}\B| \fbxo h(\tau_{1})dy-\fbxo h(\tau_{2} )dy\B|\|h_{y}\|^{\alpha}_{L^{\alpha}(Q(r))},\\
&\B|\lim_{\epsilon\rightarrow0}\int_{0}^{T}\int_{\mathbb{T}} f{\varphi} dydt\B|\leq C  \B| \fbxo h(\tau_{1})dy-\fbxo h(\tau_{2} )dy\B|\|f\|_{L^{1}(Q(r))}.
\ea\ee
Combining    \eqref{poin5} and \eqref{poin5}, we obtain the desired estimate
$$
\B| \fbxo h(\tau_{1})dy-\fbxo h(\tau_{2} )dy\B|\leq C r^{-4} \|h_{x}\|_{L^{1}(Q(r))}+Cr^{-3} \|h_{y}\|^{\alpha}_{L^{\alpha}(Q(r))}+Cr^{-1}\|f\|_{L^{1}(Q(r))},
$$
which {implies} that
\be\ba\label{poin6}
\B|h_{_{B(r),\sigma}}(\tau) -h_{_{Q(r),\sigma}}\B|=&\B|\f{1}{2r^{4}}\int_{t-r^{4}}^{t+r^{4} }\B(\fbxo h(\tau)dy-\fbxo h(\tau_{1})dy\B)d\tau_1\B|\\
\leq& C r^{-4} \|h_{y}\|_{L^{1}(Q(r))}+Cr^{-3} \|h_{y}\|^{\alpha}_{L^{\alpha}(Q(r))}+Cr^{-1}\|f\|_{L^{1}(Q(r))},
\ea\ee
where the definition  of $h_{_{Q(r),\sigma}}$ was used.

Method (3):  Formally, in the light of the equation \eqref{gsg} and integration by part, we know that
$$\ba
&\B| \fbxo h(\tau_{1})dy-\fbxo h(\tau_{2} )dy\B|\\
=&\B|\f{\int^{\tau_{2}}_{\tau_{1}}\int_{B(r)}h_{\tau}(y,\tau)\sigma d\tau dy}{\int_{B(r)}\sigma dy} \B|
\\=&\B|\f{\int^{\tau_{2}}_{\tau_{1}}\int_{B(r)} (-h_{yyyy}-\partial_{yy}|h_y|^\alpha+f)\sigma dy}{\int_{B(r)}\sigma dy} \B|\\\leq& C r^{-4} \|h_{y}\|_{L^{1}(Q(r))}+Cr^{-3} \|h_{y}\|^{\alpha}_{L^{\alpha}(Q(r))}+Cr^{-1}\|f\|_{L^{1}(Q(r))}.
\ea$$
From \eqref{poin6}, we also complete the proof of \eqref{keypoin}.
\end{proof}
 Next, we  establish some interpolation inequality for proving Theorems \ref{the1.2} and \ref{the1.3}.
\begin{lemma}\label{inter}
Suppose that $h\in L^{\infty}L^{1}(Q(r))\cap L^{2}H^{2}(Q(r))$, then
\begin{align}
\label{inter2}
&\|h_{y}\|_{L^{3}(Q(r))} \leq \|h\|^{\f13}_{L^{\infty}L^{1}(Q(r))} \|\partial_{yy}h\|^{\f23}_{L^{2}(Q(r))}
+Cr^{-\f13}\|h\|_{L^{\infty}L^{1}(Q(r))},
\\\label{inter3}
&\|h \|_{L^{3}(Q(r))}  \leq C r^{\f{4}{5}}\|h\|_{L^{1}}^{\f{11}{ 15}}\|\partial_{yy}h\|^{\f{4}{15}}_{L^{2}(Q(r))} +Cr^{\f23}\|h\|_{L^{\infty}L^{1}(Q(r))}.
\end{align}
Assume that $h$ is spatial periodic function on $(x-r_{0},x+r_{0})$ and $h\in L^{\infty}L^{2}(Q(r_{0}))\cap L^{2}H^{2}(Q(r_{0}))$, then
\be\label{inter1}
  \|h_y\|_{L^{\alpha+1}(Q(r_{0}))}     \leq C r_{0}^{\f{7-3\alpha}{2(\alpha+1)}}
    \| h \|_{L^\infty L^2 (B(r_{0}))}^{\f{\alpha+3}{4(\alpha+1)}} \|  \partial_{yy}
    h \|_{L^2(Q(r_{0}))}^{\f{3\alpha+1}{4(\alpha+1)}}.
\ee
\end{lemma}
\begin{proof}
Thanks to the  Gagliardo-Nirenberg inequality, we get
$$
\ba
&\|h_{y}\|_{L^{3}} \leq C\|h\|_{L^{1}}^{\f13}\|\partial_{yy}h\|^{\f23}_{L^{2}}+Cr^{-\f{5}{3}}{\|h\|}_{L^{1}},\\
&\|h \|_{L^{3}} \leq C  \|h\|_{L^{1}}^{\f{11}{15}}
\|\partial_{yy}h\|^{\f{4}{15}}_{L^{2}}+Cr^{-\f{2}{3}}
\|h\|_{L^{1}}.
\ea$$
Integrating over $\tau$ from $t-r^{4} $ to $t+ r^{4} $, we arrive at
$$\ba
\|h_{y}\|^{3}_{L^{3}(Q(r))} \leq& C \|h\|_{L^{\infty}L^{1}(Q(r))} \|\partial_{yy}h\|^{2}_{L^{2}(Q(r))}
+Cr^{-1}\|h\|_{L^{\infty}L^{1}(Q(r))}^{3},
\ea$$
and
$$\ba
\|h \|_{L^{3}(Q(r))}^{3} \leq& C\|h\|_{L^\infty L^{1}}^{\f{11}{ 5}}\int_{t-r^{4}}^{t+r^{4}}\|\partial_{yy}h\|^{\f{4}{ 5}}_{L^{2}}d\tau+Cr^2\|h\|_{L^{\infty}L^{1}(Q(r))}^{3}\\\leq& C\|h\|_{L^\infty L^{1}}^{\f{11}{ 5}}\B(\int_{t-r^{4}}^{t+r^{4}}\|\partial_{yy}h\|^{2}_{L^{2}}d\tau\B)^{\f{2}{ 5}}r^{\f{12}{5}}+Cr^{2}\|h\|^{3}_{L^{\infty}L^{1}(Q(r))},
\ea
$$
where the H\"older inequality was used. We get the desired estimates \eqref{inter2} and \eqref{inter3}.

Since
$\int_{x-r_{0}}^{x+r_{0}} h_ydy=0$,
 we derive from the fractional   Poincar\'e inequality  that
\be
\|h_y\|_{L^{\alpha+1}(B(r_{0}))}  \leq C \| h \|_{\dot{H}^{\f{3\alpha+1}{2(\alpha+1)}}(B(r_{0}))}.
\ee
As \cite{[OR]},
the latter inequality can be derived as follows
$$ \ba
  \|h_y\|_{L^{\alpha+1}(B(r_{0}))}^2&\leq
  C\|h_y\|_{H^{\f{\alpha-1}{2(\alpha+1)}}(B_1)}^2\leq C \sum_{k\in  \mathbb{Z}}
  \left( 1+|k|^{\f{\alpha-1}{ (\alpha+1)}}\right) \left| \widehat{h_y} (k) \right|^2 \\
  &=C \sum_{k\ne 0} \left(|k|^2+|k|^{2+\f{\alpha-1}{ (\alpha+1)}}\right)
   \left| \widehat{h} (k) \right|^2 \leq C \sum_{k\ne 0} |k|^{2+\f{\alpha-1}{ (\alpha+1)}} \left| \widehat{h} (k) \right|^2\\
  &\leq  C \| h \|_{\dot{H}^{\f{3\alpha+1}{2(\alpha+1)}}(B(r_{0}))}^2,
  \ea
$$
  where $\hat{f} (k)$ denotes the $k$-th Fourier mode
   in the Fourier expansion of $f$ on $(x-r_{0}, x+r_{0})$.

By means of interpolation inequality, we obtain
$$
  \|h_y\|_{L^{\alpha+1}(B(r_{0}))}  \leq C \| h
   \|_{\dot{H}^{ \f{3\alpha+1}{2(\alpha+1)}}(B(r_{0}))}  \leq  C
    \| h\|_{L^2 (B(r_{0}))}^{\f{\alpha+3}{4(\alpha+1)}} \|  \partial_{yy}
   h \|_{L^2(B(r_{0}))}^{\f{3\alpha+1}{4(\alpha+1)}} ,
$$
We further conclude by the H\"older inequality that
$$\ba
  \|h_y\|^{\alpha+1}_{L^{\alpha+1}(Q(r_{0}))}  \leq C
    \| h\|_{L^{\infty}L^2 (Q(r_{0}))}^{\f{\alpha+3}{4 }} \int^{t+ r^{4}_{0}}_{t-r^{4}_{0}}\|  \partial_{yy}
   h \|_{L^2(B(r_{0}))}^{\f{3\alpha+1}{4 }}d\tau,\\
   \leq C r_{0}^{\f{7-3\alpha}{2}}
    \| h\|_{L^{\infty}L^2 (Q(r_{0}))}^{\f{\alpha+3}{4 }}  \|  \partial_{yy}
   h \|_{L^2(Q(r_{0}))}^{\f{3\alpha+1}{4 }}.
\ea$$
The proof of this lemma is complete.
\end{proof}
Finally, we recall the regular estimate of biharmonic heat equation and the
general Morrey-Campanato integral characterization of H\"older spaces recently established in
\cite{[OR]}.

\begin{lemma}(\cite{[OR]})\label{lineregular}
Suppose that $0<\vartheta_{2}<1$, $0<\rho $, $H,H_x\in L^2(Q( \rho))$ and that $H$ is a distributional solution to the biharmonic heat equation $H_t=-H_{xxxx}$ in $Q( \rho)$, that is
\begin{equation}\label{limiteq}
\iint_{Q(\rho)} H \, \phi_t dxdt= \iint_{Q(\rho)} H\, \phi_{xxxx}dxdt
\end{equation}
for every $\phi\in C_0^\infty (Q( \vartheta_{2}\rho))$. Then
$$
\|H_x\|_{L^\infty(Q(\vartheta_{2}\rho))}\le C_{2}(\vartheta_{2},\rho)\left(\|H\|_{L^2(Q(\rho))}
+\|H_x\|_{L^2(Q(\rho))}\right)
$$
for some $C_{2}(\vartheta_{2},\rho)>0$.
\end{lemma}
\begin{lemma}(\cite{[OR]})\label{generalCampanatoemma}
  Let $R\in(0,1)$, $f\in L^1(Q(  R))$ and assume that there exist positive constants $\nu\in(0,1]$, $M>0$, such that
$$
\B(\ffint ~~~|f(y)-f_{_{Q(r) }}|^{p}dyds\B)^{\f1p }\leq Mr^{\nu}
$$
  for any  $z=(x,t)\in Q(R/4)$  and any $0<r<R/ 4 $. Then $f$ is H\"older continuous in $Q(R/4)$, namely, for any $z,w\in Q(R/4)$, $z=(x,t)$, $w=(y,s)$,
  \be\label{fHolder}
  |f(x,t)-f(y,s)|\le cM(|x-y|+|t-s|^{1/4})^\nu.
  \ee
\end{lemma}
\subsection{Preliminary   results for the  modified Navier-Stokes equations}
Notice  that if  the pair  $(u (x,t), \Pi(x,t)) $  solve  the
 modified Navier-Stokes system \eqref{mns},
 then the pair $(u_{\lambda}, \Pi_{\lambda}   )$ is also a solution of \eqref{mns}   for any $\lambda\in  \mathbb{R}^{+} ,$ where
\be\label{mnscaling}
 u_{\lambda}=\lambda^{\f{1}{ \alpha-1}}u(\lambda x,\lambda^{2}t),\Pi_{\lambda}=\lambda^{\f{\alpha}{ \alpha-1}}\Pi(\lambda x,\lambda^{2}t).
\ee
Based on this, we introduce some dimensionless quantities for the
 modified Navier-Stokes equations \eqref{mns} as follows
 $$
 \ba
E_{p}(u;\,r)&=\frac{1}{r^{\f{ 5\alpha-5-p}{\alpha-1}}}\iint_{Q(r)}|u(y,t)|^{p}dy dt,
& P_{\f{\alpha+1}{\alpha}}(r) &=\frac{1}{r^{\f{4\alpha-6}{\alpha-1}}}\iint_{Q(r)}{|\Pi(y,t)|^{\frac{\alpha+1}{\alpha}}}dydt,\\
E(u;\,r)&=\f{1}{r^{\f{5-3\alpha}{1-\alpha}}}\iint_{Q(r)}|\nabla u(y,t)|^2dydt,
&E_{\ast}(u;\,r)&=\sup_{-r^2\leq t<0}\f{1}{r^{\f{5-3\alpha}{1-\alpha}}}\int_{B(r)}|u(y,t)|^2dy.
 \ea
$$
{Here  we used the following notation
\begin{align*}
  &B(x;r)=\{y\in\R^3 | |x-y|\leq r\}, && B(r)=B(0;r), && \tilde{B}(r)=B(x_0;r), \\
  &Q(x,t;r)=B(x;r)\times (t-r^2,t), && Q(r)=Q(0,0;r), && \tilde{Q}(r)=Q(x_0,t_0;r).
\end{align*}
}
Now  we  present the definition of the suitable weak solution to the  modified Navier-Stokes equations \eqref{mns}.
\begin{definition}\label{defins}
A  pair  $(u,\,\Pi)$ is called a suitable weak solution to the modified Navier-Stokes equations \eqref{mns}  provided the following conditions are satisfied,
\begin{enumerate}[(1)]
\item $u \in L^{\infty}({-T,0;\,L^{2}(\Omega)})\cap L^{2}(-T,0;\,\dot{H}^{1}(\Omega)),\Pi\in
L^{\f{\alpha+1}{\alpha}}(-T,0;L^{\f{\alpha+1}{\alpha}}(\Omega)).$\label{nsSWS1}
\item$(u,~\Pi)$~solves the modified Navier-Stokes equations \eqref{mns} in ${\Omega\times (-T,0)} $ in the sense of distributions.\label{SWS2}
\item$(u, ~\Pi)$ satisfies the following inequality {for a.e. $t\in[-T,0]$}
\begin{align}
&\int_{\Omega} |u|^{2} \varphi (x,t)
+{2\int^{t}_{-T}\int_{\Omega} \varphi (x, s)
 |\nabla u |^{2}} \nonumber\\ \leq& \int^{t}_{-T}\int_{\Omega} |u|^{2}
(\partial_{t}\varphi +\Delta \varphi )
 + 2\frac{\alpha-1}{\alpha}\int^{t}_{-T}
\int_{\Omega}u\cdot\nabla\varphi  (u_{1}^{\alpha}+u_{2}^{\alpha}+u_{3}^{\alpha})+ 2\int^{t}_{-T}
\int_{\Omega}u\cdot\nabla\varphi  \Pi, \label{locnsns}
\end{align}
 			where non-negative function $\varphi(x,s)\in C_{0}^{\infty}(\Omega\times (-T,0) )$.
\end{enumerate}
\end{definition}
The pressure equation  of the  the  modified Navier-Stokes equations \eqref{mns} reads
$$\Delta\Pi=-\phi \partial_{i}\partial_{j} (u_{j} u_{i}^{\alpha-1}).$$
The usual local technique of the pressure $\Pi$ is to use the following equation
 \be\label{localpressure}
\partial_{i}\partial_{i}(\Pi\phi)=-\phi \partial_{i}\partial_{j} (u_{j} u_{i}^{\alpha-1})
+2\partial_{i}\phi\partial_{i}\Pi+\Pi\partial_{i}\partial_{i}\phi,
\ee
where $\phi$ is a standard smooth cut-off function.
In the spirit of \cite[Lemma 5.4, p.802]{[CKN]},  \cite[Lemma2.4, p.1236]{[WW]}   and    \cite[Lemma 2.2, p.11]{[RWW]}, we  can establish the following  decay estimates of dimensionless quantity  involving pressure via local pressure equation \eqref{localpressure} and the interior estimate of harmonic  function.
  We omit the detail here. We leave this to the interested readers.
\begin{lemma}\label{presurens}
For $0<\mu\leq\f{1}{8}\rho$, there exists an absolute constant $C$  independent of $\mu$ and $\rho$ such that
\be
P_{^{\f{\alpha+1}{\alpha}}}(\mu) \leq C\left(\dfrac{\rho}{\mu}\right)^{\f{ 4\alpha-6}{\alpha-1}}
 E_{^{ {\alpha+1} }}{(u;\rho)}
+C
\left(\f{\mu}{\rho}\right)^{\f{3-\alpha }{ \alpha-1 }}
P_{ {\f{\alpha+1}{\alpha}}}(\rho).
\ee
\end{lemma}
 The interpolation inequality and the Poincar\'e-Sobolev inequality
ensures the following dimensionless quantity estimates.
\begin{lemma}\label{ineqns}
For $0<\mu\leq\f{1}{2}\rho$,~
there is an absolute constant $C$  independent of  $\mu$ and $\rho$,~ such that
 \begin{align}
 E_{\alpha+1}{(u;\mu)} &\leq  C \left(\dfrac{\rho}{\mu}\right)^{\f{4\alpha-6}{\alpha-1}}
E_{\ast}^{\f{5-\alpha }{4 }}{(u;\rho)}E^{\f{3\alpha-3}{4}}{(u;\rho)}
    +C\left(\dfrac{\mu}{\rho}\right)^{\f{3-\alpha}{\alpha-1}}E_{\alpha+1}{(u;\rho)}.\label{ineq1/2}
 \end{align}
\end{lemma}

\section{Partial regularity of generalized surface growth model   }\label{section3}

\subsection{ Regularity criterion at one scale }
This section is devoted to the proof of Theorem \ref{the1.1} via blow up analysis developed in \cite{[LS],[Lin],[OR]}. It suffices to prove the following proposition.

 \begin{prop}\label{keyprop} Let $h$ be a suitable weak
 solutions of equation \eqref{gsg}  in $Q(r)$  with $f\in\mathcal{M}^{m,\f{\alpha+1}{\alpha}}(Q(r))$. There exist $\varepsilon_{01}$, $\theta\in(0,\f12)$ and $R_{1}>0$ such that    if  there holds $r<R_1$
  \be\label{firststep}
\Phi(x,t;r):=\B(\frac{1}{r^{\f{2(3-2\alpha)}{1-\alpha}}}\iint_{Q(r)}|h_y|^{\alpha+1} dyd\tau\B)^{\f{1}{\alpha+1}}<\varepsilon_{01},
  \ee
  then $h$ is  H\"older continuous at point $(x,t)$.
\end{prop}
To this end, we need the following decay lemma.
\begin{lemma}\label{lem3.1}
Let $h$ be a suitable weak solutions of \eqref{gsg} in $Q(r)$ with $f\in \mathcal{M}^{m,\f{\alpha+1}{\alpha}}(Q(r))$. For $\theta\in (0,\f12)$, there exist $ \varepsilon_{02}, R_{2}$ such that if  for some $r\leq R_{2}$,
 $$
  \B(\frac{1}{r^{\f{2(3-2\alpha)}{1-\alpha}}}\iint_{Q(r)}|h_y|^{\alpha+1}dyd\tau \B)^{\f{1}{\alpha+1}}+ \| f\|_{\mathcal{M}^{m,\f{\alpha+1}{\alpha}}(Q(r))}r^{\beta}<\varepsilon_{02},
$$
where $m>\f{5(\alpha-1)}{3\alpha-2}$ and $m\geq\f{  (\alpha+1)}{ \alpha }$, $0<\beta<\f{3\alpha-2}{\alpha-1}-\f5m$,
  then there holds the decay estimate
\be\label{d1}
  \Phi(x,t;\theta r)\leq C_{3} \theta^{\f{1}{\alpha-1}} (\Phi(x, t; r) +\| f\|_{\mathcal{M}^{m,\f{\alpha+1}{\alpha}}(Q(r))}r^{\beta}).
\ee

\end{lemma}
\begin{proof} We will prove this proposition by contradiction. Suppose there exist sequences $r_{k}\rightarrow0,\{\varepsilon_{k}\},{(x_k,t_k)}$  , $ f_{k}\in \mathcal{M}^{m,\f{\alpha+1}{\alpha}}(Q(r_{k})) $, and a sequence of suitable weak solution $\{h_{k}\}$
 such that
$$\ba
\Phi(x_k,t_{k};r_k)=&\B(\frac{1}{r_{k}^{\f{2(3-2\alpha)}{1-\alpha}}} \iint_{Q(x_k,t_{k};r_k)}|\partial_x h_{k}|^{\alpha+1}dxdt\B)^{\f{1}{\alpha+1}}+\| f_{k}\|_{\mathcal{M}^{m,\f{\alpha+1}{\alpha}}(Q(r_{k}))}r_{k}^{\beta} \\ =&\varepsilon_k\rightarrow0, \text{as} ~~k\rightarrow\infty,
\ea$$
and
$$
\B(\frac{1}{ (\theta r_k)^{  \f{2(3-2\alpha)}{1-\alpha} }}\iint_{Q(x_k,t_{k};\theta r_k)}|\partial_x h_{k}|^{\alpha+1}dxdt\B)^{\f{1}{\alpha+1}}\geq C_{3} \theta^{\f{ 1}{\alpha-1}}\varepsilon_k.
$$
To {proceed} further, we set
 $$\ba
  & H_k(x,t)= \varepsilon_k^{-1}r_{k}^{\f{\alpha-2}{1-\alpha}}\B[h_{k} \left(x_k+ x \,r_k,t_k + t\,r_k^4\right) -\f{\iint_{Q(x_k,t_{k};\vartheta_{3}r_{k})} h_{k}dxdt}{\iint _{Q(x_k,t_{k};\vartheta_{3}r_{k} )} dxdt}\B ],\\
 &  g_{k}(x,t)= \varepsilon_k^{-1} r_{k}^{\f{2-3\alpha}{1-\alpha}}f(x_k+ x \,r_k,t_k + t\,r_k^4).
\ea$$
As a consequence, there holds
\begin{align}
&\B(\iint_{\dot{Q}(1)}|\partial_xH_k|^{\alpha+1}dxds\B)^{\f{1}{\alpha+1}}+\varepsilon_{k}^{-1}{\| f_k\|}_{\mathcal{M}^{m,\f{\alpha+1}{\alpha}}(Q(r_{k}))}r_{             k}^{\beta} =1,\label{d2}\\
&  \ffgintkv ~~~  H_k dxds=0, \label{d3}\\
&\B(\iint_{\dot{Q}(\theta) }|\partial_xH_k|^{\alpha+1}dxds\B)^{\f{1}{\alpha+1}}\ge C_{3}
 \theta ^{\f{5}{\alpha+1}},\label{d4}\\
&  \iint_{\dot{Q}(1)} \left( H_{k} \, \phi_\tau - \partial_{xx}H_{k} \phi_{xx} - \varepsilon_{k}^{\alpha-1}|\partial_x H_{k}|^\alpha \phi_{xx} dxds \right)\\&= - \iint_{\dot{Q}(1)} g_{k}\phi(s)dxds, \quad  \phi \in C_0^\infty (\dot{Q}(1)).
\label{d5}
\end{align}
 and $h_k $ satisfies the local energy inequality
\be\ba\label{d6}
&\frac{1}{2}\int_{\dot{B}(1)} |H_k(t)|^2 \phi(t)dx +\int_{-1}^t \int_{\dot{B}(1)} (\partial_{xx} H_k)^2\phi \\\leq& \int_{-1}^t \int_{\dot{B}(1)} \B[\frac{1}{2}(\phi_t-\phi_{xxxx})(H_k)^2 +2(\partial_x H_k)^2\phi_{xx} \\ &
-\frac{2\alpha+1}{\alpha+1}\varepsilon_k^{ {\alpha-1} }|\partial_x H_k|^\alpha H_k\phi_x-\varepsilon_k^{ {\alpha-1} }|\partial_x H_k|^{\alpha }H_k\phi_{xx}+g_{k} H_{k}\phi\B].
\ea
\ee
By virtue of \eqref{d2}, we infer that
\be\label{d13}
\|g_{k}\|_{\mathcal{M}^{m, \f{\alpha+1}{\alpha}}(\dot{Q}(1))}\leq \varepsilon^{-1}r_{k}^{\f{2-3\alpha}{1-\alpha}}r_{k}^{-5\f{1}{m}}\|f_{k}\|_{\mathcal{M}^{m, \f{\alpha+1}{\alpha}}({Q}(r_{k}))}\leq C r_{k}^{\f{2-3\alpha}{1-\alpha}-\f5m-\beta}\leq C.
\ee
Thanks to \eqref{d13} and \eqref{d2},  Poincar\'e inequality \eqref{poin3} with $\vartheta_{1}=7/8$ and \eqref{d3} with $\vartheta_{3}=7/8$, we have
\be\label{d7}
\iint_{\dot{Q}_{(\f{7}{8})}} |H_k |^{\alpha+1} \leq C_{1}.
\ee
 Abusing  notation slightly, we denote the subsequences  of $\{H_{k}\}$ by $\{H_{k}\}$.
Now, we choose a subsequences  of $\{H_{k}\}$ such that
\be
H_{k } \rightharpoonup H,\quad \partial_x H_{k } \rightharpoonup H_x \qquad \text{ in }L^{\alpha+1} (\dot{Q}_{(\f{7}{8})})\text{ as } k\to \infty.
\ee
Let $k\rightarrow\infty$ in \eqref{d5},
$$   \iint_{\dot{Q}(1)}   (H \, \phi_\tau - \partial_{xx}H \phi_{xx})  dxdt =0,
 \quad \phi \in C_0^\infty (\dot{Q} (\f{7}{8}) )$$
 With the help of Lemma  \ref{lineregular} and \eqref{d7}, we get
 $$\ba
 \| H_x\|_{L^\infty(\dot{Q} {(\f{3}{4})})} &\leq C_{2} \left( \| H \|_{L^2 (\dot{Q }(\f{7}{8}))} + \| H_x \|_{L^2 (\dot{Q }(\f{7}{8}))}  \right) \\
&\leq C(C_{2},C_{1}).\ea$$
As a consequence,
\be\label{d11}
\f{1}{\theta^{5}}\iint_{\dot{Q}(\theta) }|\partial_x H|^{\alpha+1}\leq C(C_{2},C_{1}).
 \ee
Assume for a while we have proved that
\be \label{d12}
\partial_x H_{k } \to H_x ~\text{in} ~  L^{\alpha+1} (\dot{Q}(1/2)) ~\text{with} ~\alpha<\f{7}{3}.
\ee
 Taking the limit in \eqref{d4},  using \eqref{d12}  and \eqref{d11}, we deduce that
\be C_{3}\leq\B(\f{1}{\theta^{5}}\iint_{\dot{Q}(\theta)
}|\partial_x H_{k}|^{\alpha+1}\B)^{\alpha+1}\leq C(C_{2},C_{1}).
\ee
We take $C_{3}=2C(C_{2},C_{1})$ to get a contradiction.

It remains to show   \eqref{d12} we have assumed. To this end, we require the uniform  bound of the right hand side of  \eqref{d4}  to apply Aubin-Lions lemma.
This together with \eqref{d2} and \eqref{d4}, \eqref{d13}  leads to
\be\label{d8}
\|H_k \|_{L^\infty (  L^2(\dot{Q}(3/4)))} + \| \partial_{xx} H_k \|_{L^2(\dot{Q}(3/4))} \leq C,
\ee
which turns out that
\be\label{d9}
\|\partial_{x}H_{k}\|_{L^{\f{10}{3}} (\dot{Q}(3/4))}\leq C,
\ee
and
\be\label{d10}
\| H_{k} \|_{L^{3} ( H^{\f43}(\dot{Q}(3/4)))}\leq \|H_{k}\|^{\f13}_{L^{\infty}(L^{2}(\dot{Q}(3/4))))}\|H_{k}\|^{\f23}_{L^{2}(H^{2}(\dot{Q}(3/4))))}\leq C.
\ee
It follows from \eqref{d5} that
\be\ba\label{d111}
&\left| \iint_{ Q(3/4)}\partial_t H_k \,\phi dxdt\right|\\
=& \left| -\iint_{Q(3/4)} \partial_{xx} H_k \, \phi_{xx}dxdt - \varepsilon_k^{\alpha-1}  \iint_{\dot{Q}(3/4)}|\partial_x h_k  |^{ \alpha } \phi_{xx}+g_{k}\phi dxdt\right|\\
\leq& ( \| \partial_{xx} H_k  \|_{L^{\f{\alpha+1}{\alpha} } (\dot{Q}(3/4))} +\| \partial_{x} H_k  \|^{\alpha}_{L^{\alpha+1} (\dot{Q}(3/4))}  +\|g_{k}\|_{\mathcal{M}^{m, \f{\alpha+1}{\alpha}}(\dot{Q}(1))})\\&\times\| \phi \|_{L^{\alpha+1} ( W^{2,\alpha+1}(\dot{Q}(3/4)))} \\
\leq& C \| \phi \|_{L^{\alpha+1} ( W^{2,\alpha+1}(\dot{Q}(3/4)))},
\ea
\ee
for all $\phi \in C_0^\infty (\dot{Q}(3/4))$.\\
Therefore,  there holds $\|\partial_t H_k\|_{L^{\f{\alpha+1}{\alpha } ( I_{3/4} ;(W^{2,\alpha+1}(\dot{Q}(3/4)))^*)}} \leq C $.
 Since
\be
H^{\f43}\subset H^{\f16}\subset (W^{2,\alpha+1})^{\ast}
\ee
 Aubin--Lions   lemma allows us to select  a subsequence of $\{H_k\}$ that converges in $L^{3} ( H^{7/6}(\dot{Q}(3/4)))$. Sobolev embedding theorem
 helps us to get that $\partial_x H_k $ converges in $L^{3} (\dot{Q}(3/4))$ .\\
 From \eqref{d9}, we infer that
 $$\partial_x H_{k_n} \to H_x ~\text{in} ~  L^{\alpha+1} (\dot{Q}(1/2)) ~\text{with} ~\alpha<\f73.$$
Thus, we give the proof of assertion \eqref{d12}.
The proof of this lemma is completed.
 \end{proof}
At this stage,
   iterating the   above lemma and using  the general parabolic  Campanato Lemma \ref{generalCampanatoemma} allow  us to prove Proposition \ref{keyprop}.
\begin{proof}[Proof of Proposition \ref{keyprop}]
We set
$R_{1}=min\B\{R_{2},\big[\f{\varepsilon_{02}}{4 \| f\|_{ \mathcal{M}^{m,\f{\alpha+1}{\alpha}} (Q(r))} }\big]^{ \frac{1}{\beta} }\B\}$, therefore,
$\forall (x_{0},t_{0})\in Q(x,t;\rho)$ with $\rho\leq\f{r}{2}$, there holds
$$
\B(\frac{1}{(\f{r}{2})^{\f{2(3-2\alpha)}{1-\alpha}}}\iint_{Q(x_0,t_0;\f{r}{2})}|h_x|^{\alpha+1} \B)^{\f{1}{\alpha+1}}\leq  \B(\frac{4^{\f{3-2\alpha}{1-\alpha}}}{(r)^{\f{2(3-2\alpha)}{1-\alpha}}}\iint_{Q(x_0,t_0;\f{r}{2})}|h_x|^{\alpha+1} \B)^{\f{1}{\alpha+1}}\leq 4^{\f{3-2\alpha}{1-\alpha^{2}}}\varepsilon_{01}.
$$
We choose   $\varepsilon_{01}$ sufficiently small such that
\be\ba\label{3.20}
&\B(\frac{1}{(\f{r}{2})^{\f{2(3-2\alpha)}{1-\alpha}}}\iint_{Q(x_0,t_0;\f{r}{2})}|h_x|^{\alpha+1} \B)^{\f{1}{\alpha+1}}+ \| f\|_{{\mathcal{M}^{m,\f{\alpha+1}{\alpha}}}(\tilde{Q}(r/2))} \B(\f{r}{2}\B)^{\beta}<\f12\varepsilon_{02}.
\ea\ee
We claim that,
for
\be
0<\beta_{1}<\f{1}{\alpha-1} ~~ \text{and}~~  \beta_{1}\leq\beta,
\ee
there holds
\be\label{claims}\left\{\ba
&\Phi(x_{0},t_{0};\f{r}{2}\theta^{k-1})+\| f\|_{M^{\f{5}{4-\gamma},\f{\alpha+1}{\alpha}}(\tilde{Q}(\f{r}{2}\theta^{k-1}))} (\f{r}{2}\theta^{k-1})^{\beta}\leq\varepsilon_{02} ;\\
&\Phi(x_{0},t_{0};\f{r}{2}\theta^{k})\leq \theta^{k\beta_{1}}\B[\Phi(x_{0},t_{0};\f{r}{2})+\| f\|_{{\mathcal{M}^{m,\f{\alpha+1}{\alpha}}}(\tilde{Q}(r/2))} (\f{r}{2})^{\beta}\B].
\ea  \right.\ee
We will prove this assertion  by induction arguments.   The case $k=1$ is a direct application
of Lemma \ref{lem3.1}.   To
illustrate the   ideas in the proof of \eqref{claims} for $k=k$, we  also present the detail of \eqref{claims} for $k=2$. Then, we assume that there holds  \eqref{claims} with $k=k-1$  and    prove \eqref{claims} for $k=k$.

To proceed further, we take $\theta$   sufficiently small such that
\be\label{3.23}
2C_{3}\theta^{\f{1}{2}(\f{1}{\alpha-1}-\beta_1)}<1.
\ee
First, with \eqref{3.20} in hand, we can invoke  Lemma \ref{lem3.1} to obtain
 $$\ba \Phi(x_{0},t_{0};\f{r}{2} \theta )\leq & C_{3}\theta^{\f{1}{\alpha-1}}\B[\Phi(x_{0},t_{0};\f{r}{2})+\| f\|_{ \mathcal{M}^{m,\f{\alpha+1}{\alpha}} (\tilde{Q}(r/2))} (\f{r}{2})^{\beta}\B]\\
=&C_{3}\theta^{\f{1}{2}(\f{1}{\alpha-1}-\beta_1)}\theta^{\f{1}{2}(\f{1}{\alpha-1}+\beta_1)}
\B[\Phi(x_{0},t_{0};\f{r}{2})+\| f\|_{ \mathcal{M}^{m,\f{\alpha+1}{\alpha}} (\tilde{Q}(r/2))} (\f{r}{2})^{\beta}\B]\\
\leq& \theta^{\beta_1}\B[\Phi(x_{0},t_{0};\f{r}{2})+\| f\|_{ \mathcal{M}^{m,\f{\alpha+1}{\alpha}} (\tilde{Q}(r/2))} (\f{r}{2})^{\beta}\B],
\ea$$
where we have used \eqref{3.23}.  This together with \eqref{3.20} implies that
we have proved \eqref{claims} with $k=1.$

Second, we use $ \eqref{claims}_{2}$ with $k=1 $ and \eqref{3.20}
to get
$$\ba
\Phi(x_{0},t_{0};\f{r}{2}\theta )+\| f\|_{ \mathcal{M}^{m,\f{\alpha+1}{\alpha}} (\tilde{Q}(\f{r}{2}\theta ))} (\f{r}{2}\theta )^{\beta}\leq&
 \theta^{\beta_1}\B[\Phi(x_{0},t_{0};\f{r}{2})+\| f\|_{ \mathcal{M}^{m,\f{\alpha+1}{\alpha}} (\tilde{Q}(r/2))} (\f{r}{2})^{\beta}\B]\\&+\| f\|_{ \mathcal{M}^{m,\f{\alpha+1}{\alpha} }(\tilde{Q}(\f{r}{2}\theta ))} (\f{r}{2}\theta )^{\beta}\\
\leq &2\theta^{\beta_1}\B[\Phi(x_{0},t_{0};\f{r}{2})+\| f\|_{ \mathcal{M}^{m,\f{\alpha+1}{\alpha}} (\tilde{Q}(r/2))} (\f{r}{2})^{\beta}\B]\\
\leq &\theta^{\beta_1}\varepsilon_{02}.
\ea$$
By means of  Lemma \ref{lem3.1}  again and \eqref{3.23}, we conclude that
\be\ba\label{3.24}
\Phi(x_{0},t_{0};\f{r}{2}\theta^{2})\leq& C_{3}\theta^{\f{1}{\alpha-1}}\B[\Phi(x_{0},t_{0};\f{r\theta}{2})+\| f\|_{ \mathcal{M}^{m,\f{\alpha+1}{\alpha}} (\tilde{Q}(r/2))} (\f{r\theta}{2})^{\beta}\B]\\
\leq& C_{3}\theta^{\f{1}{2}(\f{1}{\alpha-1}-\beta_1)}\theta^{\f{1}{2}(\f{1}{\alpha-1}+\beta_1)}\B\{\theta^{\beta_1}\big[\Phi(x_{0},t_{0};\f{r}{2})+\| f\|_{ \mathcal{M}^{m,\f{\alpha+1}{\alpha}} (\tilde{Q}(r/2))} (\f{r}{2})^{\beta}\big]\\&+\| f\|_{ \mathcal{M}^{m,\f{\alpha+1}{\alpha}} (\tilde{Q}(r/2))} (\f{r\theta}{2})^{\beta}\B\}\\
\leq& C_{3}\theta^{\f{1}{2}(\f{1}{\alpha-1}-\beta_1)}
\theta^{\f{1}{2}(\f{1}{\alpha-1}+\beta_1)}\theta^{\beta_1}
2\B[\Phi(x_{0},t_{0};\f{r}{2})+\| f\|_{ \mathcal{M}^{m,\f{\alpha+1}{\alpha}} (\tilde{Q}(r/2))} (\f{r}{2})^{\beta}\B]\\
\leq& \theta^{2\beta_1}
\B[\Phi(x_{0},t_{0};\f{r}{2})+\| f\|_{ \mathcal{M}^{m,\f{\alpha+1}{\alpha}} (\tilde{Q}(r/2))} (\f{r}{2})^{\beta}\B].
\ea\ee
As a consequence, we show \eqref{claims} with $k=2.$

Third, we assume that \eqref{claims} is valid for $k=k-1$.  With the help of $ \eqref{claims}_{2}$ with $k=k-1 $ and \eqref{3.20}, we infer that
$$\ba
&\Phi(x_{0},t_{0};\f{r}{2}\theta^{k-1} )+\| f\|_{ \mathcal{M}^{m,\f{\alpha+1}{\alpha}} (\tilde{Q}(\f{r}{2}\theta^{k-1}  ))} (\f{r}{2}\theta^{k-1}  )^{\beta}\\
\leq&
 \theta^{\beta_1(k-1)}\B[\Phi(x_{0},t_{0};\f{r}{2} )+\| f\|_{ \mathcal{M}^{m,\f{\alpha+1}{\alpha}} (\tilde{Q}(r/2))} (\f{r}{2}\theta^{k-1})^{\beta}\B]\\&+\| f\|_{ \mathcal{M}^{m,\f{\alpha+1}{\alpha}} (\tilde{Q}(\f{r}{2}\theta ))} (\f{r}{2}\theta^{k-1} )^{\beta}\\
\leq &2\theta^{\beta_1(k-1)}\B[\Phi(x_{0},t_{0};\f{r}{2})+\| f\|_{ \mathcal{M}^{m,\f{\alpha+1}{\alpha}} (\tilde{Q}(r/2))} (\f{r}{2})^{\beta}\B]\\
\leq &\theta^{\beta_1(k-1)}\varepsilon_{02}.
\ea$$
Arguing the same manner as \eqref{3.24}, we have
$$\ba
\Phi(x_{0},t_{0};\f{r}{2}\theta^{k})\leq& C_{3}\theta^{\f{1}{\alpha-1}}\B[\Phi(x_{0},t_{0};\f{r\theta^{k-1}}{2})+\| f\|_{ \mathcal{M}^{m,\f{\alpha+1}{\alpha}} (\tilde{Q}( r/2))} (\f{r\theta^{k-1}}{2})^{\beta}\B]\\
\leq& C_{3}\theta^{\f{1}{2}(\f{1}{\alpha-1}-\beta_1)}
\theta^{\f{1}{2}(\f{1}{\alpha-1}+\beta_1)}
\B\{\theta^{(k-1)\beta_1}\B[\Phi(x_{0},t_{0};\f{r}{2})+\| f\|_{ \mathcal{M}^{m,\f{\alpha+1}{\alpha}} (\tilde{Q}(r/2))} (\f{r}{2})^{\beta}\B]\\&+\| f\|_{ \mathcal{M}^{m,\f{\alpha+1}{\alpha}} (\tilde{Q}(r/2))} (\f{r\theta^{k-1}}{2})^{\beta}\B\}\\
\leq& C_{3}\theta^{\f{1}{2}(\f{1}{\alpha-1}-\beta_1)}\theta^{\f{1}{2}(\f{1}{\alpha-1}+\beta_1)}\theta^{\beta_1(k-1)}
2\B[\Phi(x_{0},t_{0};\f{r}{2})+\| f\|_{ \mathcal{M}^{m,\f{\alpha+1}{\alpha}} (\tilde{Q}(r/2))} (\f{r}{2})^{\beta}\B]\\
\leq& \theta^{k\beta_1}
\B[\Phi(x_{0},t_{0};\f{r}{2})+\| f\|_{ \mathcal{M}^{m,\f{\alpha+1}{\alpha}} (\tilde{Q}(r/2))} (\f{r}{2})^{\beta}\B].
\ea$$
Collecting the above estimates, we derive that \eqref{claims} for $k=k$. Hence, it is   shown that  \eqref{claims} is valid.

Now, it follows from $\eqref{claims}_{2}$  that,
$\forall \rho\in (0,\f{r}{2}), $
\be\label{3.25}
\Phi(x_{0},t_{0};{\rho})\leq C\f{\rho^{\beta_{1}}}{(\f{r}{2})^{\beta_{1}}}\B[\Phi(x_{0},t_{0};\f{r}{2})+\| f\|_{ \mathcal{M}^{m,\f{\alpha+1}{\alpha}} (\tilde{Q}(r/2))} \big(\f{r}{2}\big)^{\beta}\B].
\ee
For any $\varrho\in (0,\f{r}{4})$, there exist $\rho$ such $\varrho<2\varrho<\f{r}{2}$. Therefore, employing  Lemma \ref{parapoincare} with $\vartheta_{1}=\f12,$ \eqref{3.25},  and $\alpha>1$, we observe that
 \be\ba\label{3.26}
&\B (\ffgintrtietarou ~~~ \B|h(t)-\ffgintrtietarou~~ hdxdt\B|^{\alpha+1}dxdt\B)^{\f{1}{\alpha+1}}\\
\leq&  C(2\varrho)^{\f{(\alpha-2)}{\alpha-1}}
\B\{\f{1}{(2\varrho)^{\f{2(3-2\alpha)}{1-\alpha}}}
\iint_{\tilde{Q}  (2\varrho) }|h_{x}|^{\alpha+1}dxdt
+\B[\f{1}{(2\varrho)^{\f{2(3-2\alpha)}{1-\alpha}}}
\iint_{\tilde{Q}(2\varrho)}|h_{x}|^{\alpha+1}dxdt\B]^{\alpha}\B\}
^{\f{1}{\alpha+1}}\\&+C(2\varrho)^{4-\f{5}{m}} \| f\|_{\mathcal{M}^{m,\f{\alpha+1}{\alpha}}(\tilde{Q}(2\varrho))} \\
\leq&  C(2\varrho)^{\f{(\alpha-2)}{\alpha-1}}
\f{\rho^{\beta_{1}}}{(\f{r}{2})^{\beta_{1}}}
\B[\Phi(x_{0},t_{0};\f{r}{2})+\| f\|_{ \mathcal{M}^{m,\f{\alpha+1}{\alpha}} (\tilde{Q}(r/2))} \big(\f{r}{2}\big)^{\beta}\B]+C(2\varrho)^{4-\f{5}{m}} \| f\|_{\mathcal{M}^{m,\f{\alpha+1}{\alpha}}(\tilde{Q}(r/2))} \\
\leq&  C \varrho ^{\beta_{2}},
 \ea\ee
where $\beta_{2}=\min\{\f{(\alpha-2)}{\alpha-1}+\beta_{1},4-\f{5}{m}\}$.
If $\alpha\geq2$, it is clear that $\beta_{2}>0$. If $1<\alpha<2$, the choice of $\beta_{1} $ and $\beta_{2}$ (that is, $\f{1}{\alpha-1}+\f{(\alpha-2)}{\alpha-1}>0$ and $\f{3\alpha-2}{\alpha-1}-\f5m +\f{(\alpha-2)}{\alpha-1}{>0}$) also ensures that $\beta_{2}>0$.  From  \eqref{3.26} and Lemma \ref{generalCampanatoemma}, we complete the proof of this Proposition \ref{keyprop}.
\end{proof}

 \subsection{Regularity criterion at  infinite scales } \label{section4}
\begin{proof}[Proof of Theorem \ref{the1.2}]
The condition   \eqref{con2} yields
that there exists  a positive constant $ r_{0} $ such that
\be\label{condition1}
E(r)\leq \varepsilon_{02},~\text{for any} ~~r\leq r_{0}.
\ee
It follows from the local energy inequality \eqref{loc} and the H\"older inequality that
\be\ba\label{4.1}
&E_{\ast}(r)+E(r)\\\leq &C\B[D_{2}(2r)+E_{2}(2r)+E_{\alpha+1}(2r)+E^{\f{\alpha}{\alpha+1}}_{\alpha+1}(2r) D^{\f{1}{\alpha+1}}_{\alpha+1}(2r)+r^{{\f{-2\alpha^{2}+6\alpha-2}{\alpha^{2}-1}}}\|f\|_{L^{\f{\alpha+1}{\alpha}}(Q(2r))}D_{\alpha+1}^{\f{1}{\alpha+1}}(2r) \B]\\
\leq &C\B[D^{\f{2}{\alpha+1}}_{\alpha+1}(2r)+E^{\f{2}{\alpha+1}}_{\alpha+1}(2r)+E_{\alpha+1}(2r)+E^{\f{\alpha}{\alpha+1}}_{\alpha+1}(2r) D^{\f{1}{\alpha+1}}_{\alpha+1}(2r)+r^{ \f{2-3\alpha}{1-\alpha}-\f5m }\|f\|_{\mathcal{M}^{m,\f{\alpha+1}{\alpha}}}D_{\alpha+1}^{\f{1}{\alpha+1}}(2r) \B].
\ea\ee
Since $h-C$ is also the solution of \eqref{gsg}, we can replace $h$ by $h-h_{_{Q(r) }}$  or $ h-h_{_{Q(r),\sigma}}$  in \eqref{4.1}. We will take $h-h_{_{Q(r) }}$ in \eqref{4.1}.  It is worth remarking that the following proof  works for both  $h-h_{_{Q(r) }}$  and $ h-h_{_{Q(r),\sigma}}$. Hence, we reformulate inequality \eqref{4.1} as
\be\ba\label{4.2}
&\tilde{E}_{\ast}(r)+E(r)\\
\leq & C\B[\tilde{D}^{\f{2}{\alpha+1}}_{\alpha+1}(r)+E^{\f{2}{\alpha+1}}_{\alpha+1}(r)+E_{\alpha+1}(2r)+E^{\f{\alpha}{\alpha+1}}_{\alpha+1}(2r) \tilde{D}^{\f{1}{\alpha+1}}_{\alpha+1}(2r)\\&+  r^{ \f{{3\alpha-2}}{\alpha -1}-\f{5}{m}  } \|f\|_{\mathcal{M}^{m,\f{\alpha+1}{\alpha}}(Q(2r))}\tilde{D}_{\alpha+1}^{\f{1}{\alpha+1}}(2r)\B]\\
 \ea\ee
 Taking advantage of Lemma \ref{parapoincare}, we have
\be\label{4.3}
\tilde{D}_{\alpha+1}(r) \leq E_{\alpha+1}(4r)+E^{\alpha}_{\alpha+1}(4r)+r^{(\alpha+1)
(\f{2-3\alpha}{1-\alpha}-\f5m)}\|f\|^{\alpha+1}_{\mathcal{M}^{m,1}(Q(2r))}.
\ee
 From the interpolation inequality \eqref{inter1}, we see that
 \be\label{4.4}
 E_{\alpha+1}(r)\leq \tilde{E}_{\ast}^{\f{\alpha+3}{8 }}(r)E^{\f{3\alpha+1}{8 }}(r), E_{\alpha+1}^{\f{2}{\alpha+1}}(r)\leq \tilde{E}_{\ast}^{\f{ (\alpha+3)}{4(\alpha+1) }}(r)E^{\f{ (3\alpha+1)}{4(\alpha+1) }}(r).
 \ee
 Combining \eqref{4.3} and \eqref{4.4}, we know that
\be\ba\label{4.5}
 \tilde{D}_{\alpha+1}^{\f{1}{\alpha+1}}(r) \leq& C E_{\alpha+1}^{\f{1}{\alpha+1}}(4r)+CE^{\f{ \alpha}{\alpha+1}}_{\alpha+1}(4r)+Cr^{(\f{2-3\alpha}{1-\alpha}-\f5m)}\|f\|_{\mathcal{M}^{m,1}(Q(2r))}\\
\leq& C\big[\tilde{E}_{\ast}^{\f{\alpha+3}{8 }}(2r)E^{\f{3\alpha+1}{8 }}(2r)\big]^{\f{1}{\alpha+1}}+C\big[E_{\ast}^{\f{\alpha+3}{8 }}(2r)E^{\f{3\alpha+1}{8 }}(2r)\big]^{\f{ \alpha}{\alpha+1}} \\&+Cr^{(\f{2-3\alpha}{1-\alpha}-\f5m)}\|f\|_{\mathcal{M}^{m,1}(Q(2r))}\\
\leq& C [E_{\ast}^{\f{\alpha+3}{8(\alpha+1) }}(2r)E^{\f{3\alpha+1}{8 (\alpha+1) }}(2r) + C \tilde{E}_{\ast}^{\f{\alpha(\alpha+3)}{8(\alpha+1) }}(2r)E^{\f{\alpha(3\alpha+1)}{8(\alpha+1) }}(2r) \\&+Cr^{(\f{2-3\alpha}{1-\alpha}-\f5m)}\|f\|_{\mathcal{M}^{m,1}(Q(2r))}],
 \ea\ee
 and
  \be\label{4.66}
 \tilde{D}^{\f{2}{\alpha+1}}_{\alpha+1}(r)\leq C \B[\tilde{E}_{\ast}^{\f{\alpha+3}{4(\alpha+1) }}(2r)E^{\f{3\alpha+1}{4 (\alpha+1) }}(2r) + E_{\ast}^{\f{\alpha(\alpha+3)}{4(\alpha+1) }}(2r)E^{\f{\alpha(3\alpha+1)}{4(\alpha+1) }}(2r) +r^{2(\f{2-3\alpha}{1-\alpha}-\f5m)}
 \|f\|^{2}_{\mathcal{M}^{m,1}(Q(2r))}\B].
 \ee
As a consequence, we arrive at
  \be\ba\label{4.6}
 E^{\f{\alpha}{\alpha+1}}_{\alpha+1}(2r) \tilde{D}^{\f{1}{\alpha+1}}_{\alpha+1}(2r)\leq&
C \tilde{E}_{\ast}^{\f{(\alpha+3) }{8  }}(2r)E^{\f{(3\alpha+1) }{8   }}(2r) + C\tilde{E}_{\ast}^{\f{\alpha(\alpha+3)}{4(\alpha+1) }}(2r)E^{\f{\alpha(3\alpha+1)}{4(\alpha+1) }}(2r) \\&+Cr^{(\f{2-3\alpha}{1-\alpha}-\f5m)}\|f\|_{M^{m,1}}\tilde{E}_{\ast}^{\f{\alpha(\alpha+3)}{8(\alpha+1) }}(r)E^{\f{\alpha(3\alpha+1)}{8(\alpha+1) }}(r).\ea\ee
Plugging {\eqref{4.4}, \eqref{4.5}}, \eqref{4.66} and \eqref{4.6} into \eqref{4.2}, we end up with
\be\ba\label{4.7}
&\tilde{E}_{\ast}(r)+E(r)\\\leq &C\B[\tilde{D}^{\f{2}{\alpha+1}}_{\alpha+1}(r)+E^{\f{2}{\alpha+1}}_{\alpha+1}(r)+E_{\alpha+1}(2r)+E^{\f{\alpha}{\alpha+1}}_{\alpha+1}(2r) \tilde{D}^{\f{1}{\alpha+1}}_{\alpha+1}(2r)+r^{{ \f{{3\alpha-2}}{\alpha -1}-\f{5}{m}  }}\|f\|\tilde{D}_{\alpha+1}^{\f{1}{\alpha+1}}(2r)\B]\\
\leq& C \B[\tilde{E}_{\ast}^{\f{\alpha+3}{4(\alpha+1) }}(2r)E^{\f{3\alpha+1}{4 (\alpha+1) }}(2r) + \tilde{E}_{\ast}^{\f{\alpha(\alpha+3)}{4(\alpha+1) }}(2r)E^{\f{\alpha(3\alpha+1)}{4(\alpha+1) }}(2r) +r^{2(\f{2-3\alpha}{1-\alpha}-\f5m)}\|f\|^{2}_{\mathcal{M}^{m,1}(Q(2r))}\B]\\
&+C\tilde{E}_{\ast}^{\f{ (\alpha+3)}{4(\alpha+1) }}(r)E^{\f{ (3\alpha+1)}{4(\alpha+1) }}(2r)+C\tilde{E}_{\ast}^{\f{\alpha+3}{8 }}(2r)E^{\f{3\alpha+1}{8 }}(2r)\\
&+
C \tilde{E}_{\ast}^{\f{(\alpha+3) }{8  }}(2r)E^{\f{(3\alpha+1) }{8   }}(2r) + C\tilde{E}_{\ast}^{\f{\alpha(\alpha+3)}{4(\alpha+1) }}(2r)E^{\f{\alpha(3\alpha+1)}{4(\alpha+1) }}(2r) \\&+Cr^{(\f{2-3\alpha}{1-\alpha}-\f5m)}\|f\|_{\mathcal{M}^{m,1}(Q(2r))}
\tilde{E}_{\ast}^{\f{\alpha(\alpha+3)}{8(\alpha+1) }}(r)E^{\f{\alpha(3\alpha+1)}{8(\alpha+1) }}(r)\\
&+Cr^{ \f{2-3\alpha}{1-\alpha}-\f5m }\|f\|_{\mathcal{M}^{m,\f{\alpha+1}{\alpha}}(Q(2r))}
 \tilde{E}_{\ast}^{\f{\alpha+3}{8(\alpha+1) }}(2r)E^{\f{3\alpha+1}{8 (\alpha+1) }}(2r) \\&+C r^{ \f{2-3\alpha}{1-\alpha}-\f5m  }\|f\|_{\mathcal{M}^{m,\f{\alpha+1}{\alpha}}(Q(2r))}\tilde{E}_{\ast}^{\f{\alpha(\alpha+3)}{8(\alpha+1) }}(2r)E^{\f{\alpha(3\alpha+1)}{8(\alpha+1) }}(2r) \\&+Cr^{2(\f{2-3\alpha}{1-\alpha}-\f5m)} \|f\|_{\mathcal{M}^{m,\f{\alpha+1}{\alpha}}(Q(2r))}\|f\|_{\mathcal{M}^{m,1}(Q(2r))}.
 \ea\ee
 Since $\f{\alpha+1}{\alpha}>\f{5(\alpha-1)}{3\alpha-2}$ with $1<\alpha<7/3$, there holds $r^{ \f{2-3\alpha}{1-\alpha}-\f5m  } \|f\|_{_{\mathcal{M}^{m,1}(Q(2r))}}\rightarrow0$
 as $r\rightarrow0$.

Notice that $1<\alpha<7/3$  guarantees $\max\{\f{\alpha+3}{4(\alpha+1)},\f{\alpha(\alpha+3)}{4(\alpha+1) }, \f{(\alpha+3) }{8  },\f{\alpha+3}{8(\alpha+1) },\f{\alpha(\alpha+3)}{8(\alpha+1) }\}<1$, in addition,
in view of \eqref{4.7},    the iteration method as \cite{[CKN]} together with  \eqref{condition1} helps us  to   get the smallness of $\tilde{E}_{\ast}(r)+E(r)$ for $0<r<r_{2}<r_{1} $.
The interpolation inequality implies  the smallness of $ E_{\alpha+1}(r)$. We conclude the proof by Theorem \ref{the1.1}.
\end{proof}
\subsection{Improvement of  regularity criterion }
\begin{proof}[Proof of Theorem \ref{the1.3}]
First, we focus on the proof of \eqref{impro1}.
We rewrite \eqref{inter1} with $\alpha=2$  and $f=0$ in Lemma 2.2 as \be\label{4.11}
 E_{3}(4r)\leq C \tilde{E}_{\ast}^{\f58}(4r)E^{\f18}(4r)E^{\f78}_{\f{12}{7},2}(4r).
\ee
With \eqref{4.11} in hand, arguing in the same manner as \eqref{4.5}, we observe that
\be\label{4.12}
\tilde{D}^{\f13}_{3}(2r) \leq E^{\f13}_{3}(4r)+E^{\f23}_{3}(4r)
\leq  (E_{\ast}^{\f58}(4r)E^{\f18}(4r)E^{\f78}_{\f{12}{7},2}(4r))^{\f13}+ (E_{\ast}^{\f58}(4r)E^{\f18}(4r)E^{\f78}_{\f{12}{7},2}(4r))^{\f23}
\ee
From \eqref{4.2} with $\alpha=2$ and $f=0$, we conclude that
\be\ba\label{4.13}
\tilde{E}_{\ast}(r)+E(r)\leq &C\B[\tilde{D}^{\f{2}{3}}_{3}(2r)+E^{\f{2}{3}}_{3}(2r)
+E_{3}(2r)+E^{\f{2}{3}}_{3}(2r) \tilde{D}^{\f{1}{3}}_{3}(2r) \B].
\ea\ee
Plugging \eqref{4.11} and \eqref{4.12} into \eqref{4.13},  we infer that
\be\ba
&\tilde{E}_{\ast}(r)+E(r)\\
\leq &C \B[E_{\ast}^{\f58}(4r)E^{\f18}(4r)E^{\f78}_{\f{12}{7},2}(4r)\B]^{\f23}+ C \B[E_{\ast}^{\f58}(4r)E^{\f18}(4r)E^{\f78}_{\f{12}{7},2}(4r)\B]^{\f43} \\&
+
\B[E_{\ast}^{\f58}(4r)E^{\f18}(4r)E^{\f78}_{\f{12}{7},2}(4r)\B]^{\f23}+C
  E_{\ast}^{\f58}(4r)E^{\f18}(4r)E^{\f78}_{\f{12}{7},2}(4r) \\&
+C\B[E_{\ast}^{\f58}(4r)E^{\f18}(4r)E^{\f78}_{\f{12}{7},2}(4r)\B]^{\f23}\B[ (E_{\ast}^{\f58}(4r)E^{\f18}(4r)E^{\f78}_{\f{12}{7},2}(4r))^{\f13}+ (E_{\ast}^{\f58}(4r)E^{\f18}(4r)E^{\f78}_{\f{12}{7},2}(4r))^{\f23}\B]\wred{.}
\ea\ee
Set $F(r)=\tilde{E}_{\ast}(r)+E(r)$, hence,
$$\ba
F(r)\leq &F^{\f12}(4r) E_{\f{12}{7},2}^{\f{7}{12}} (4r)+F (4r) E_{\f{12}{7},2}^{\f{7}{6}} (4r)
+F^{\f34} (4r) E_{\f{12}{7},2}^{\f{7}{6}} (4r)\\
&+F^{\f12}(4r) E_{\f{12}{7},2}^{\f{7}{12}} (4r)
\B[F^{\f{1}{4}}(4r) E_{\f{12}{7},2}^{\f{7}{24}} (4r)+F^{\f12}(4r) E_{\f{12}{7},2}^{\f{7}{12}} (4r)\B]\\
\leq & CF^{\f12}(4r) E_{\f{12}{7},2}^{\f{7}{12}} (4r)+CF (4r) E_{\f{12}{7},2}^{\f{7}{6}} (4r)
+CF^{\f34} (4r) E_{\f{12}{7},2}^{\f{7}{6}} (4r).
\ea$$
This together with  iteration method mentioned above implies the smallness of $F(r)$ under the smallness of ${E_{\f{12}{7},2}}(4r)$ . Combining this with
Theorem \ref{the1.1}, we complete the proof of this part.

We turn our attention to the proof of \eqref{impro2}.
From Lemma \ref{inter},  we see that
\be\ba\label{4.9}
&D_{3}(r)\leq C D^{{\f{11}{5}}}_{\infty,1}(r)E^{\f25}(r)+C D^{3}_{\infty,1}(r),\\
&E_{3}(r)\leq C D_{\infty,1}(r)E (r)+C D^{3}_{\infty,1}(r).
\ea\ee
Substituting \eqref{4.1} with $\alpha=2$ and $f=0$ into \eqref{4.9}, we infer that
$$\ba
&D_{3}(r)\leq C D^{{\f{11}{5}}}_{\infty,1}(r)[D^{\f{2}{3}}_{3}(2r)
+E^{\f{2}{3}}_{3}(2r)+E_{3}(2r)+E^{\f{2}{3}}_{3}(2r) D^{\f{1}{3}}_{3}(2r)]^{\f25}(r)+C D^{3}_{\infty,1}(r)\\
&E_{3}(r)\leq C D_{\infty,1}(r)[D^{\f{2}{3}}_{3}(2r)
+E^{\f{2}{3}}_{3}(2r)+E_{3}(2r)+E^{\f{2}{3}}_{3}(2r) D^{\f{1}{3}}_{3}(2r)]+C D^{3}_{\infty,1}(r).
\ea$$
Hence,
\be\ba\label{4.10}
D_{3}(r)+E_{3}(r)\leq& C D^{{\f{11}{5}}}_{\infty,1}(r)[D^{\f{2}{3}}_{3}(2r)
+E^{\f{2}{3}}_{3}(2r)+E_{3}(2r)+E^{\f{2}{3}}_{3}(2r) D^{\f{1}{3}}_{3}(2r)]^{\f25} \\
&+ C D_{\infty,1}(r)[D^{\f{2}{3}}_{3}(2r)
+E^{\f{2}{3}}_{3}(2r)+E_{3}(2r)+E^{\f{2}{3}}_{3}(2r) D^{\f{1}{3}}_{3}(2r)]+C D^{3}_{\infty,1}(r).
\ea\ee
Before going further, we write
$$G(\mu)=D_{3}(\mu)+E_{3}(\mu).$$
As a consequence,  we get
$$G(r)\leq  C D^{{\f{11}{5}}}_{\infty,1}(r)[G^{\f{2}{3}}(2r)
 +G_{3}(2r) ]^{\f25}  + C D_{\infty,1}(r)[G^{\f{2}{3}}(2r)
 +G_{3}(2r) ]+C D^{3}_{\infty,1}(r).$$
 An iteration argument leads to the smallness of $G(r)$ under the  the smallness of $D^{3}_{\infty,1}(r)$. With this in hand, Theorem \ref{the1.1} entails us to  achieve
 the proof of this case.
\end{proof}
\section{Partial regularity of  the modified Navier-Stokes equations   }
In this section,  we study the partial regularity of suitable weak solutions of  the modified Navier-Stokes equations \eqref{mns}.
In step 1,  in the spirit of blow up technique in  \cite{[HLW],[GWZ],[HW],[Wang]},  we  prove  Theorem \ref{the1.1ns} by the fractional integration theorem \cite{[OLeary],[Kukavica0]}  involving parabolic Morrey spaces.
  Step 2 is   devoted to  optimal Hausdorff dimension of possible singular point set $\mathcal{S}$ in the equations \eqref{mns}.
\subsection{ Proof of Theorem \ref{the1.1ns} }
\begin{lemma}
There exist $\varepsilon_{1}>0$  and $\theta\in (0,\f12)$ such that if $(u,\Pi)$ is a suitable weak
solutions of the modified Navier-Stokes equations in $Q(1)$ satisfying
$$
\B (\iint_{Q(1)} |u|^{\alpha+1}dxdt\B)^{\f{1}{\alpha+1}}
+\B(\iint_{Q (1)}
|\Pi |^{\f{\alpha+1}{\alpha}}dxdt \B)^{\f{\alpha}{\alpha+1}}  \leq    \varepsilon_{1},$$
then
$$\ba
&\B(\theta^{\f{6-4\alpha}{\alpha-1}}\iint_{Q(\theta)} |u|^{\alpha+1}dxdt\B)^{\f{1}{\alpha+1}}
+\B(\theta^{\f{6-4\alpha}{\alpha-1}}\iint_{Q (\theta)}
|\Pi |^{\f{\alpha+1}{\alpha}}dxdt \B)^{\f{\alpha}{\alpha+1}}  \\
\leq&     \f12\B (\iint_{Q(1)} |u|^{\alpha+1}dxdt\B)^{\f{1}{\alpha+1}}
+\B(\iint_{Q (1)}
|\Pi |^{\f{\alpha+1}{\alpha}}dxdt \B)^{\f{\alpha}{\alpha+1}}.
\ea$$
\end{lemma}
\begin{proof}
We suppose that the statement is invalid. Then, for any $\theta\in(0,\f12)$, there exists  a sequence    of
suitable weak  solutions  of the modified Navier-Stokes  equations and a sequence $\varepsilon_{1k}$
such that
$$\ba
&\B (\iint_{Q(1)} |u_{k}|^{\alpha+1}dxdt\B)^{\f{1}{\alpha+1}}
+\B(\iint_{Q (1)}
|\Pi_{k} |^{\f{\alpha+1}{\alpha}}dxdt \B)^{\f{\alpha}{\alpha+1}} =    \varepsilon_{1k}\rightarrow0\; \text{as}\; k\rightarrow\infty,\\
&\B (\theta^{\f{6-4\alpha}{\alpha-1}}\iint_{ Q(\theta) } |u|^{\alpha+1}dxdt\B)^{\f{1}{\alpha+1}}
+\B(\theta^{\f{6-4\alpha}{\alpha-1}}\iint_{ Q(\theta) }
|\Pi |^{\f{\alpha+1}{\alpha}}dxdt \B)^{\f{\alpha}{\alpha+1}} >\f12 \varepsilon_{1k}.
\ea$$
Before going further, we write
$$
(v_{i})_{k}=\varepsilon_{1k}^{-1}(u_{i})_{k},\pi_{k}=\varepsilon_{1k}^{-1}\Pi_{k}, i=1,2,3.
$$
Hence, we get
\be\ba\label{mdns4.2}
&\B (\iint_{Q(1)} |v_{k}|^{\alpha+1}dxdt\B)^{\f{1}{\alpha+1}}
+\B(\iint_{Q (1)}
|\pi_{k} |^{\f{\alpha+1}{\alpha}}dxdt \B)^{\f{\alpha}{\alpha+1}} = 1,\\
&\B (\theta^{\f{6-4\alpha}{\alpha-1}}\iint_{ Q(\theta) } |v_{k}|^{\alpha+1}dxdt\B)^{\f{1}{\alpha+1}}
+\B(\theta^{\f{6-4\alpha}{\alpha-1}}\iint_{ Q(\theta) }
|\pi_{k} |^{\f{\alpha+1}{\alpha}}dxdt \B)^{\f{\alpha}{\alpha+1}} >\f12.
\ea\ee
In addition, there holds
\be\label{mdns4.3}
\iint_{Q(1)}\B[-v_{k}\partial_{\tau}\phi-v_{k}\Delta\phi-\varepsilon_{1k}^{\alpha-1}
 \sum_{i=1}^{3}(v_{i})_{k}^{\alpha-1}v\cdot\nabla\phi_i -
\pi_{k} \mathrm{div}\phi \B] dxd\tau=0,
\ee
and the local energy inequality below
\be\label{mdns4.4}\ba
&\int|v_{k}|^{2} \varphi  dx+2 \int_{-T}^{t} \int_{B(1)}|\nabla v_{k}|^{2} \varphi \\\leq&
 \int_{-T}^{t} \int_{B(1)} |v_{k}|^{2}(\varphi_{t}+\Delta\varphi )+2\varepsilon_{1k}^{\alpha-1}\f{\alpha-1}{\alpha}
 \int_{-T}^{t} \int_{B(1)} v_{k}\cdot\nabla \varphi\B[(v_{1})^{\alpha}_{k}+(v_{2})^{\alpha}_{k}+(v_{3})^{\alpha}_{k}\B]\\&
+2 \int_{-T}^{t} \int_{B(1)} v_{k}\cdot\nabla \varphi\pi_{k} dxdt.
\ea\ee
We conclude by \eqref{mdns4.2} that there exist the subsequences of $v_{k}$ and $\pi_{k}$ satisfying
$$v_{k}\rightharpoonup v ~~\text{in}~~ L^{\alpha+1}(Q(1)), \pi_{k}\rightharpoonup \pi~~ \text{in} ~~ L^{\f{\alpha+1}{\alpha}}(Q(1)).$$
In addition,  the lower semicontinuity ensures that
\be\label{mdns4.5}
\|v\|_{L^{\alpha+1}(Q(1))}+\|\pi\|_{L^{\f{\alpha+1}{\alpha}}(Q(1))}\leq C.
\ee
Combining the Holder inequality and \eqref{mdns4.4}, we arrive at
\be\label{mdns4.6}
\|v_{k}\|_{L^{\infty}(L^{2})(Q(\f78))}+\|\nabla v_{k}\|_{L^{2}(L^{2})(Q(\f78))}\leq C.\ee
As a consequence, taking $k\rightarrow\infty$ in \eqref{mdns4.3}, we infer that
$$\iint_{Q(1)}\B[-v\partial_{\tau}\phi -v\Delta\phi-\pi\text{div}\phi\B] dxdt=0.$$
Thanks to the regularity of the Stokes equations and \eqref{mdns4.5}, we know that
$$
\theta^{\f{6-4\alpha}{\alpha-1}}\iint_{Q(\theta)} |v|^{\alpha+1}dxdt\leq
\theta^{5+\f{6-4\alpha}{\alpha-1}}\iint_{Q(1)} |v|^{\alpha+1}dxdt \leq C
\theta^{ \f{1+\alpha}{\alpha-1}}. $$
Since the  pressure equations of $\pi_{k}$ is given by
$$
\Delta\pi_{k}=- \varepsilon_{1k}^{\alpha-1} \partial_{i}\partial_{j}
[(v_{j})_{k}(v_{i})_{k}^{\alpha-1}],
$$
  by a slightly modification the derivation of  Lemma  \ref{presurens}  and \eqref{mdns4.2},  we  further get
$$\ba
&\theta^{\f{6-4\alpha}{\alpha-1}}\iint_{Q (\theta)}
| \pi_{k}  |^{\f{\alpha+1}{\alpha}}dxdt\\
\leq& C \varepsilon_{1k}^{\f{\alpha^{2}-1}{\alpha}}\theta^{\f{6-4\alpha}{\alpha-1}} \iint_{Q(1)} |v_{k}|^{\alpha+1}dxdt+C \theta^{\f{3-\alpha}{\alpha-1}}\iint_{Q (1)}
| \pi_{k}  |^{\f{\alpha+1}{\alpha}}dxdt\\
\leq& C \varepsilon_{1k}^{\f{\alpha^{2}-1}{\alpha}} \theta^{\f{6-4\alpha}{\alpha-1}}+
C\theta^{\f{3-\alpha}{\alpha-1}},
\ea$$
that is
$$\B(\theta^{\f{6-4\alpha}{\alpha-1}}\iint_{Q (\theta)}
| \pi_{k} |^{\f{\alpha+1}{\alpha}}dxdt\B)^{\f{\alpha}{\alpha+1}}\leq C  \varepsilon_{1k}^{ \alpha-1} \theta^{\f{\alpha(6-4\alpha)}{\alpha^{2}-1}}+
C\theta^{\f{\alpha(3-\alpha)}{\alpha^{2}-1}}.$$
Substituting this into  \eqref{mdns4.2},  we observe that
\be\label{mdns4.7}\f12<\B (\theta^{\f{6-4\alpha}{\alpha-1}}\iint_{ Q(\theta) } |v_{k}|^{\alpha+1}dxdt\B)^{\f{1}{\alpha+1}}
+C  \varepsilon_{1k}^{ \alpha-1} \theta^{\f{\alpha(6-4\alpha)}{\alpha^{2}-1}}+
C\theta^{\f{\alpha(3-\alpha)}{\alpha^{2}-1}} .
\ee
It follows from the H\"older inequality and \eqref{mdns4.6} that
$$\ba
&\B|\iint_{Q(\f78)} {v_{k}}_{t} \phi  dxdt\B|\\=&\B|-\iint_{Q(1)}\nabla v_{k} \nabla\phi+\varepsilon_{1k}^{\alpha-1}
\sum_{i=1}^{3}(v_{i})_{k}^{\alpha-1}v_{k}\cdot\nabla\phi_i  +\pi_{k} \mathrm{div}\phi \B|\\
 \leq&(\|\nabla v_{k}\|_{L^{2}(Q(\f78))}+\| |v_k|^{\alpha} \|_{\f{\alpha+1}{\alpha}(Q(\f78))}
 +\|\pi_{k}\|_{\f{\alpha+1}{\alpha}(Q(\f78))})
 \|\nabla\phi\|_{L^{\alpha+1}(Q(\f78))} \\
\leq& \|\ \phi \|_{L^{\alpha+1}(W^{1,\alpha+1}(B(\f78)))}.
\ea$$
 Hence, we find that
 $$\|{v_{k}}_{t}\|_{ L^{\f{\alpha+1}{\alpha}} ((W^{1,\alpha+1})^{\ast} )}\leq C.$$
This together with  $\|  v_{k}\|_{L^{2}(H^{1})}\leq C$   and the classical Aubin--Lions   lemma, we have
$$v_{k}\rightarrow v ~~\text{in}~~ L^{2}(Q(\f78)).$$
 Notice that \eqref{mdns4.6} implies that $\|  v_{k}\|_{L^{\f{10}{3}}( Q(\f78))}\leq C$, we obtain
 $$v_{k}\rightarrow v ~~\text{in} ~~ L^{\alpha+1}(Q(\f78)), \alpha<\f73.$$
 Passing the limit in \eqref{mdns4.7}, we know that
$$\f12< C
\theta^{ \f{1 }{\alpha-1}}+C \varepsilon_{k}^{ \alpha-1}\theta^{\f{\alpha(6-4\alpha)}{\alpha^{2}-1}}+
C\theta^{\f{\alpha(3-\alpha)}{\alpha^{2}-1}}.
$$
First choosing $\theta$ sufficiently small and then taking $k$ sufficiently large, we get a contradiction.
\end{proof}

\begin{proof}[ Proof of Theorem \ref{the1.1ns} ]
With the aid of Lemma 4.1, arguing as the
same manner as in \eqref{3.25}, we have,
for any $ (x,t)\in Q(x,t;\rho) $ with $\rho\leq\f{1}{2}$, there exist a constant
$0<\beta_{3}<1$ such that
$$\ba
&\B(\rho^{\f{6-4\alpha}{\alpha-1}}\iint_{Q(\rho)} |u|^{\alpha+1}dxdt\B)^{\f{1}{\alpha+1}}
+\B(\rho^{\f{6-4\alpha}{\alpha-1}}\iint_{Q (\rho)}
|\Pi |^{\f{\alpha+1}{\alpha}}dxdt \B)^{\f{\alpha}{\alpha+1}}
\leq C\rho^{\f{\alpha\beta_{3}}{\alpha+1}}\wred{\varepsilon_1}.
\ea$$
This means that  $\Pi, |u|^{\alpha} \in \mathcal{M}^{\f{5(\alpha+1)}{\alpha(\f{\alpha+1}{\alpha-1}-\beta_{3})},
\f{\alpha+1}{\alpha}}(Q(\f{1}{2}))$ and  breaks the scaling of the equations.  Then one can apply   the fractional integration theorem \cite{[OLeary],[Kukavica0]} ( Riesz potential estimate \cite{[HLW],[GWZ],[HW],[Wang]} ) involving
parabolic Morrey spaces to get that,   for any $  q<\infty$, $ u \in  L^{q}(Q(\f{1}{4}))$.
The desired  boundness of $|v|$
can be improved by bootstrapping arguments (see \cite{[HLW],[DHW]}).
 The rigorous proof can be found in these works.
 Here we just outline the proof of $ u \in   L^{q}(Q(\f{1}{4})).$
The fractional integration theorem due to \cite{[OLeary],[Kukavica0]}  reads
$$\B\|\iint_{\mathbb{R}^{3}} \f{f}{(|x-\xi|+\sqrt{|t-\tau|})^{4}}d\xi d\tau\B\|_{L^{p}(\Omega\times I)}\leq
\|f\|_{L^{m}(\Omega\times I)}^{\f{m}{p}}\|f\|_{\mathcal{M}^{\f{5q}{5-\lambda},q}(\Omega\times I)}^{1-\f{m}{p} },
$$
with
$$\f{1}{p}>\f{q}{m}(\f1q-\f{1}{5-\lambda}). $$
Roughly speaking, notice that
$$|u| \approx \iint \f{|u|^{\alpha}+|\Pi|}{(|x-\xi|+\sqrt{|t-\tau|})^{4}}d\xi d\tau,$$
then, applying the fractional integration theorem mentioned above, we see that
$u\in L^{p}$ with $$   p<\f{1}{\f{\alpha}{m}(1-\f{q}{\alpha(5-\lambda)})}=\f{m}{\alpha-\f{q}{5-\lambda}}.$$
We start with
$q\equiv\alpha+1, m_{0}=\alpha+1, $ to obtain
$$
p_{1}<\f{\alpha+1}{\f{1+\alpha-\beta_{3}(\alpha-1)\alpha}
{1+\alpha-\beta_{3}(\alpha-1) }}. $$
Then, we set
$m_{1}=p_{1}$ to derive that
$$\ba
&p_{2}<\f{m_{1}}{\f{1+\alpha-\beta_{3}(\alpha-1)\alpha}
{1+\alpha-\beta_{3}(\alpha-1) }}.\ea$$
Likewise,
$$\ba
&p_{k}<\f{m_{k}}{\f{1+\alpha-\beta_{3}(\alpha-1)\alpha}
{1+\alpha-\beta_{3}(\alpha-1) }}=\f{p_{k-1}}{\f{1+\alpha-\beta_{3}(\alpha-1)\alpha}
{1+\alpha-\beta_{3}(\alpha-1) }}.
\ea$$
We see that
$$
\alpha+1<p_{1}<p_{2}<\cdots<p_{k}<\cdots <\infty.
$$
Therefore,
for any $\alpha+1\leq q<\infty$, we get $u\in L^{p}$.
The proof of this theorem is completed.
\end{proof}
\subsection{Proof of Theorem \ref{the1.2ns}}
Next, we turn our attentions to the proof of Theorem \ref{the1.2ns}.  From the  hypothesis of this theorem, we know that there exist a constant
  $r_{0}$ such that  $E_{\ast}(u;\,r)\leq\varepsilon_{2}$   for $r\leq r_{0}$.
It seems that, unlike the Navier-Stokes equations,  the pressure equations of  the  modified  Navier-Stokes equations
  $$\Delta\Pi=-\phi \partial_{i}\partial_{j} ((u_{j}-C) (u_{i}^{\alpha-1}-C )),$$
  is invalid.
  To this end, by means of the local energy inequality and the  hypothesis, we first to show  the smallness of  $ E_{\alpha+1}(u;\,r)$  for $0<r \leq r_1<r_{0}$. Then, using the local energy inequality once again, we can complete the proof by Theorem \ref{the1.1ns}.
\begin{proof}[Proof of Theorem \ref{the1.2ns}]By means of the usual test function and H\"older's inequality, we derive from the local energy inequality \eqref{locnsns}  that
\be\label{eq4.6}
 \begin{aligned}
 E(u;\,r) +E_{\ast}(u;\,r)
 \leq& C\Big[E_{2}(u;\,2r) +E_{\alpha+1}(u;\,2r)
+P_{\f{\alpha+1}{\alpha}}^{\f{\alpha}{\alpha+1}}(2r)
E_{\alpha+1}^{\f{1}{\alpha+1}}(u;\,2r) \Big].
 \end{aligned}
\ee
Multiplying (\ref{eq4.6}) by $\varepsilon^{3/16}$ and using the  H\"older and  Young inequality, we see that
\be\label{nseq4.7}
 \begin{aligned}
&\varepsilon^{3/16}E(u;\,r) +\varepsilon^{3/16}E_{\ast}(u;\,r)\\
\leq& C\varepsilon^{3/16}
[E_{2}(u;\,2r) +E_{\alpha+1}(u;\,2r)
+P_{\f{\alpha+1}{\alpha}}^{\f{\alpha}{\alpha+1}}(u;\,2r)
E_{\alpha+1}^{\f{1}{\alpha+1}}(u;\,2r) ]\\\leq& C
({E_{\alpha+1}(u;\,2r)})^{\f{2}{\alpha+1}}\varepsilon^{3/16}+
C\varepsilon^{3/16}
[ E_{\alpha+1}(u;\,2r)
+P_{\f{\alpha+1}{\alpha}}^{\f{\alpha}{\alpha+1}}(u;\,2r)
E_{\alpha+1}^{\f{1}{\alpha+1}}(u;\,2r) ]\\
\leq& C \varepsilon^{\f{3(\alpha+1)}{16(\alpha-1)}} +CE_{\alpha+1}(u;\,2r) +
\varepsilon^{\f{3(\alpha+1)}{16\alpha}}
P_{\f{\alpha+1}{\alpha}}(2r).
 \end{aligned}
\ee
From Lemma \ref{ineqns} and Lemma \ref{presurens} with $\mu=2r$, we infer that
\be\ba\label{ns4.7}
&E_{\alpha+1}(u;\,2r)  \leq  C \left(\dfrac{\rho}{r}\right)^{\f{4\alpha-6}{\alpha-1}}
E^{\f{5-\alpha }{4 }}(u;\,\rho)E^{\f{3\alpha-3}{4}} _{\ast}(u;\,\rho)
+C\left(\dfrac{r}{\rho}\right)^{\f{3-\alpha}{\alpha-1}}E_{\alpha+1}(u;\,\rho), \\
&  \varepsilon^{1/4}P_{^{\f{\alpha+1}{\alpha}}}(2 r) \leq C\varepsilon^{1/4}\left(\dfrac{\rho}{r}\right)^{\f{ 4\alpha-6}{\alpha-1}}
 E_{^{ {\alpha+1} }}(u;\,\rho)
+C
\left(\f{r}{\rho}\right)^{\f{3\alpha-1}{\alpha(\alpha-1)}}
\varepsilon^{1/4}P_{ {\f{\alpha+1}{\alpha}}}(\rho).
\ea  \ee
Before going further, we write
\be\label{eq4.8}
G(r):=
E_{\alpha+1}(u;\,r) +\varepsilon^{3/16} E(u;\,r) +\varepsilon^{3/16}E_{\ast}(u;\,r)
+\varepsilon^{1/4}P_{\f{\alpha+1}{\alpha}}(r).
\ee
Plugging     \eqref{nseq4.7}  and \eqref{ns4.7}    into (\ref{eq4.8}), we arrive at
\be\label{eq4.9}\ba
G( r)\leq&C \varepsilon^{\f{3(\alpha+1)}{16(\alpha-1)}} +CE_{\alpha+1}(u;\,2r) +C
\varepsilon^{\f{3(\alpha+1)}{16\alpha}}
P_{\f{\alpha+1}{\alpha}}(2r) +C\varepsilon^{1/4}P_{\f{\alpha+1}{\alpha}}(r)\\
%\leq&C \varepsilon^{\f{3(\alpha+1)}{16(\alpha-1)}} +CE_{\alpha+1}(u;\,2r) +C\varepsilon^{1/4}P_{\f{\alpha+1}{\alpha}}(2r)\\
{\leq}& C \varepsilon^{\f{3(\alpha+1)}{16(\alpha-1)}}+C \left(\dfrac{\rho}{r}\right)^{\f{4\alpha-6}{\alpha-1}}
E^{\f{5-\alpha }{4 }}(u;\,\rho)E^{\f{3\alpha-3}{4}} _{\ast}(u;\,\rho)
    +C\left(\dfrac{r}{\rho}\right)^{\f{3-\alpha}{\alpha-1}}
    E_{\alpha+1}(u;\,\rho) \\ &+
 C\left(\dfrac{\rho}{r}\right)^{\f{ 4\alpha-6}{\alpha-1}}
 \varepsilon^{\f{1}{4}}E_{^{ {\alpha+1} }}(u;\,\rho)
   +C
\left(\f{r}{\rho}\right)^{\f{3\alpha-1}{\alpha(\alpha-1)}}
\varepsilon^{1/4}P_{ {\f{\alpha+1}{\alpha}}}(\rho).
 \ea
\ee

 Since $\f{5-\alpha }{4 }<1$, one can apply the   iteration method as \cite{[CKN]} to show that  there exist a positive constant $r_{1}$, such that
 $ E_{\alpha+1}(u;\,r) \leq\varepsilon$ for any $r\leq r_{1}.$

Using  the local energy inequality (\ref{eq4.6}) and the decay estimates \eqref{ns4.7}, we  get
\be\ba\label{eq4.10}
 &E(u;\, r) +E_{\ast}(u;\, r)+ P_{\f{\alpha+1}{\alpha}}(r)\\
\leq& C (E_{\alpha+1  }(u;\,2r))^{\f{2}{\alpha+1}} +CE_{\alpha+1}(u;\,2r)
+ CP_{\f{\alpha+1}{\alpha}}(2r)\\
\leq& C [\left(\dfrac{\rho}{r}\right)^{\f{4\alpha-6}{\alpha-1}}
E_{\alpha+1}(u;\,\rho)]^{\f{2}{\alpha+1}}+
C\left(\dfrac{\rho}{r}\right)^{\f{4\alpha-6}{\alpha-1}}E_{\alpha+1}(u;\,\rho)\\&+ C\left(\dfrac{\rho}{r}\right)^{\f{ 4\alpha-6}{\alpha-1}}
 E_{^{ {\alpha+1} }}(u;\,\rho)
   +C
\left(\f{r}{\rho}\right)^{\f{3\alpha-1}{\alpha(\alpha-1)}}
 P_{ {\f{\alpha+1}{\alpha}}}(\rho) .
\ea\ee
Now, we can apply the iteration argument as above once again to show that, there exists a positive constant  $r_{2}\leq r_{1}$ such that $ E(u;\,r) +E_{\ast}(u;\,r)+ P_{\f{\alpha+1}{\alpha}}(r)\leq \varepsilon_{0}$ for any $ r\leq r_{2}$. Finally, we can invoke  Theorem \ref{the1.1ns}  to  finish the proof of   Theorem \ref{the1.2ns}.
\end{proof}

\subsection{Proof of Theorem \ref{the1.3ns}}
The key estimate for the proof of Theorem \ref{the1.3ns} is to show
\be\label{keythe1.3ns}
\f{d}{dt}\int_{{\Omega}}|\nabla u|^{2}dx\leq   C\|u \|^{\f{2q(\alpha-1)}{q-3\alpha+3}}_{L^{q }({\Omega})}\|  \nabla u\|^{2}_{L^{ 2} ({\Omega} )}, q> 3\alpha-3.
\ee
Taking $q=6$ in the latter inequality and using the Sobolev inequality, we get
$$
\f{d}{dt}\int_{{\Omega}}|\nabla u|^{2}dx\leq   C \|  \nabla u\|^{\f{2+2\alpha}{3-\alpha}} _{L^{ 2} ({\Omega} )},
$$
which leads to the local well-posedness of initial data in $H^{1}(\Omega)$. In addition, the Gronwall inequality and estimate \eqref{keythe1.3ns} mean the rest result of this theorem.
\begin{proof}[Proof of  estimate  \eqref{keythe1.3ns}]
Multiplying the modified Navier-Stokes system (\ref{mns}) by $\partial_{k}\partial_{k}u$, integrating over $\Omega$,  using  div\,$u$=0, and integrating by
parts, we have
\be\ba\label{2.18}
\f12\f{d}{dt}\int_{\Omega}|\nabla u|^{2}dx+\int_{\Omega}|\nabla^{2}u|dx&=
\sum^{3}_{i=1}\int_{\Omega}u\cdot\nabla u_{i}^{\alpha-1} \partial_{k}\partial_{k}{u_i dx}.
\ea\ee
{From H\"older's inequality, we have}
 $$
 {I=}\sum^{3}_{i=1}\int_{\Omega}u\cdot\nabla u_{i}^{\alpha-1} \partial_{k}\partial_{k}{u_i dx}
 \leq C\||u|^{\alpha-1}\|_{L^{\f{q}{\alpha-1}}(\Omega)}
 \| \nabla u \|_{L^{\f{2q}{q-2\alpha+2}}(\Omega)}\| \nabla^{2} u \|_{L^{2}(\Omega)}.
 $$
It follows from  the interpolation  inequality  and  Sobolev's embedding     that
\be\label{sampleinterplation}
 \| \nabla u \|_{L^{\f{2q}{q-2\alpha+2}}(\Omega)}\leq
\| \nabla  u \|^{ \f{3\alpha-3}{q}}_{L^{6 }(\Omega)}\|  \nabla u\|^{ \f{q+3-3\alpha}{q}}_{L^{ 2} (\Omega)}\leq C\| \nabla^{2} u \|^{ \f{3\alpha-3}{q}}_{L^{2}(\Omega)}\|  \nabla u\|^{ \f{q+3-3\alpha}{q}}_{L^{ 2} (\Omega)}.
\ee
We derive from {Young's inequality}   and the latter inequality that
\be\ba\label{i11}
I&\leq {C\|u_{}\|}^{\alpha-1}_{L^{q }(\Omega)}\| \| \nabla^{2} u \|^{1+\f{3\alpha-3}q}_{L^{2}(\Omega)}\|  \nabla u\|^{ \f{q+3-3\alpha}{q}}_{L^{ 2} ( \Omega )}\\
&\leq C\|u \|^{\f{2q(\alpha-1)}{q-3\alpha+3}}_{L^{q }(\Omega)}\|  \nabla u\|^{2}_{L^{ 2}  (\Omega)  }+\f{1}{32}\| \nabla^{2} u \|^{2}_{L^{2}(\Omega)},
\ea\ee
{where the fact} that  $q> 3\alpha-3$ were used.

Collecting above estimates and  absorbing the terms containing $\| \nabla^{2} u \|^{2}_{L^{2}(\Omega)} $   by the left-hand
side in \eqref{2.18}, we deduce that
\be\label{3.14}
\f{d}{dt}\int_{\Omega}|\nabla u|^{2}dx\leq   C\|u \|^{\f{2q(\alpha-1)}{q-3\alpha+3}}_{L^{q }(\Omega)}\|  \nabla u\|^{2}_{L^{ 2} (\Omega)}.
\ee
The desired estimate is derived.
 \end{proof}

\section{ Conclusion}\label{section5}
Inspired by the work of Oz\'anski-Robinson \cite{[OR]} and  Stein-Winkler \cite{[SW]},
we follow the path of  \cite{[OR]}
to study the
    partial regularity of suitable weak solution  of  a surface growth  equation \eqref{gsg}  with the general nonlinear term and no-zero force. The   precise relationship between the Hausdorff dimension of
potential singular point set $\mathcal{S}$ and the parameter $\alpha$ in this equation is presented in Corollary   \ref{coro}. The  {index range}  of Corollary   \ref{coro} {is restricted to} the torus {due to} the interpolation inequality \eqref{inter1}.

{Now we give some remarks:}

1.  Partial regularity of  the 3D stochastic Navier-Stokes equations was obtained by Flandoli by Romito in \cite{[FR]}.
Can one obtain      the     partial regularity of the original equation  \eqref{ckpz} or \eqref{gsg}  with the noise term?

2.  It is an interesting question to prove the existence of suitable weak solution and related partial regularity theory for the critical case $\alpha=7/3$ in \eqref{gsg} as well as the 3D  modified Navier-Stokes system \eqref{mns}.  A probable approach involving existence is to use the strategy introduced by Wu in \cite{[Wu]}.

3.  Is there any other approach to study the     partial regularity of equation \eqref{gsg}  besides blow
up {analysis?}

4.   Weather the same conclusion   of Corollary   \ref{coro}  is valid  on bounded
domains and the whole {space?} The Corollary  heavily relies on the interpolation inequality on  periodic   boundary conditions.

5.  What is optimal hypothesis of force to obtain the partial regularity of equation in  \eqref{ckpz} or \eqref{gsg}. Similar research for the 3D Navier-Stokes can be found in \cite{[Kukavica]}.

\section*{Acknowledgements}
The authors would like to express their sincere gratitude to Dr. Xiaoxin Zheng at Beihang University
for the discussion of  the modified Navier-Stokes equations \eqref{mns}.
 Wang was partially supported by  the National Natural
 Science Foundation of China under grant (No. 11971446 and  No.  11601492).
Wu was partially supported by the National Natural Science Foundation of China under grant No. 11771423.
Zhou was partially supported by the National Natural Science Foundation of China under grant No. 12071113.

\end{document}